\documentclass[12pt,letterpaper,titlepage,reqno]{amsart}
\usepackage{amsmath, amssymb, amsthm, amsfonts,amscd,amsaddr,enumerate}

\usepackage[
    paper=a4paper,
    portrait=true,
    textwidth=425pt,
    textheight=650pt,
         tmargin=3cm,
    marginratio=1:1
            ]{geometry}

\theoremstyle{plain}
\newtheorem{theorem}{Theorem}[section]
\newtheorem{proposition}[theorem]{Proposition}

\newtheorem{prop}[theorem]{Proposition}
\newtheorem{lemma}[theorem]{Lemma}

\theoremstyle{definition}

\newtheorem{rmk}[theorem]{Remark}
\numberwithin{equation}{section}
\newtheorem*{theoremA*}{Theorem A}
\newtheorem*{theoremB*}{Theorem B}
\newtheorem*{theorem1*}{Theorem A'}
\newtheorem*{theoremC*}{Theorem C}
\newtheorem*{theoremD*}{Theorem D}
\newtheorem*{theoremE*}{Theorem E}
\newtheorem*{theoremF*}{Theorem F}
\newtheorem*{theoremE2*}{Theorem E2}
\newtheorem*{theoremE3*}{Theorem E3}
\newcommand{\bs}{\backslash}

\newcommand{\C}{\mathbb{C}}

\newcommand{\Hc}{\mathcal{H}}

\newcommand{\Z}{\mathbb{Z}}

\newcommand{\Sc}{\mathcal{S}}
\newcommand{\Oc}{\mathcal{O}}

\newcommand{\Mf}{\mathfrak{M}}

\newcommand{\R}{\mathbb{R}}
\newcommand{\N}{\mathbb{N}}

\newcommand{\Sl}{\operatorname{SL}}

\newcommand{\Ind}{\operatorname{Ind}}
\newcommand{\id}{\operatorname{id}}
\newcommand{\SO}{\operatorname{SO}}
\newcommand{\PW}{\operatorname{PW}}

\newcommand{\Hom}{\operatorname{Hom}}
\newcommand{\End}{\operatorname{End}}

\newcommand{\OO}{\operatorname{O}}
\newcommand{\GL}{\operatorname{GL}}
\newcommand{\SL}{\operatorname{SL}}

\newcommand{\SU}{\operatorname{SU}}

\newcommand{\im}{\operatorname{Im}}

\newcommand{\Lie}{\operatorname{Lie}}
\newcommand{\Ad}{\operatorname{Ad}}

\newcommand{\ad}{\operatorname{ad}}
\newcommand{\diag}{\operatorname{diag}}

\newcommand{\Spec}{\operatorname{spec}}

\newcommand{\Sym}{\operatorname{Sym}}

\newcommand{\re}{\operatorname{Re}}

\def\cP{\mathcal{P}}
\def\hat{\widehat}
\def\af{\mathfrak{a}}

\def\gf{\mathfrak{g}}

\def\ef{\mathfrak{e}}

\def\kf{\mathfrak{k}}

\def\mf{\mathfrak{m}}
\def\nf{\mathfrak{n}}

\def\sf{\mathfrak{s}}

\def\so{\mathfrak{so}}

\def\su{\mathfrak{su}}
\def\tf{\mathfrak{t}}

\def\zf{\mathfrak{z}}
\def\la{\langle}
\def\ra{\rangle}
\def\1{{\bf1}}

\def\U{\mathcal{U}}
\def\Ac{\mathcal{A}}

\def\Cc{\mathcal{C}}
\def\D{\mathcal {D}}
\def\Ic{\mathcal {I}}

\def\Oc{\mathcal{O}}
\def\P{\mathbb{P}}

\def\oline{\overline}
\def\F{\mathcal{F}}

\def\tilde{\widetilde}

\def\Sym{\mathrm{Sym}}
\def\tilde{\widetilde}
\def\Diff{\mathbb{D}}

\def\oline{\overline}
\def\la{\langle}
\def\ra{\rangle}

\title[Paley-Wiener theorem]
{A Paley-Wiener theorem for Harish-Chandra modules}
\begin{document}

\begin{abstract}
We formulate and prove a Paley-Wiener theorem for Harish-Chandra modules for a real reductive group. As a corollary we obtain a new and elementary proof of the Helgason conjecture.
\end{abstract}

\author[Gimperlein]{Heiko Gimperlein}
\address{Maxwell Institute for Mathematical Sciences and Department of Mathematics, Heriot--Watt University \\ Edinburgh, EH14 4AS, United Kingdom \\
 {\tt h.gimperlein@hw.ac.uk}}

\author[Kr\"otz]{Bernhard Kr\"otz}
\address{Institute f\"ur Mathematik\\
Universit\"at Paderborn\\Warburger Str. 100,
D-33098 Paderborn \\ {\tt bkroetz@gmx.de}}

\author[Kuit]{Job J.~Kuit}
\address{Institut f\"ur Mathematik\\
Universit\"at Paderborn\\Warburger Str. 100,
D-33098 Paderborn \\ {\tt jobkuit@math.upb.de}}

\author[Schlichtkrull]{Henrik Schlichtkrull}
\address{Department of Mathematics\\
University of Copenhagen\\Universitetsparken 5 \\
DK-2100 Copenhagen \O\\
{\tt schlicht@math.ku.dk}}

\maketitle
\section{Introduction}
Let $G$ be a real reductive algebraic group and $K\subset G$ a maximal compact
subgroup. Let $V$ be a Harish-Chandra module for $(\gf, K)$ where $\gf=\Lie(G)$.
\par Every Harish-Chandra module admits a completion (globalization) to a representation of $G$. Such a completion is in general not unique. First and foremost is the smooth completion $V^\infty$ of moderate growth,
due to Casselman-Wallach, which is unique up to isomorphism, see \cite{C}, \cite[Sect. 11]{W2}
and \cite{BK}. Another completion is the $G$-module $V^{\omega}$ of analytic vectors in $V^{\infty}$ with its natural compact-open topology.

Define the minimal
completion  of $V$ by the convolution product
$$ V_{\rm min} :=C_c^\infty(G)*V\subset V^\infty$$
and endow it with a topology as follows: take a finite dimensional subspace $V_f\subset V$
which generates $V$, and consider the surjective map
$$C_c^\infty(G)\otimes V_f \twoheadrightarrow V_{\rm min}\,.$$
The quotient topology on $V_{\rm min}$ does not depend on the choice of the finite dimensional
generating subspace $V_f$ and thus induces a natural quotient
Hausdorff locally convex topology on $V_{\rm min}$.
It is inherent in the construction that $V_{\rm min}$ embeds equivariantly and
continuously into every completion of $V$, hence the terminology.

Next we review Schmid's interpretation \cite{S} of the Helgason conjecture. The conjecture was stated in \cite{H1} and first proven in \cite{K6}.
Let $\chi$ be a character of the algebra $\Diff(G/K)$ of $G$-invariant differential operators on $G/K$. The Helgason conjecture states that the Poisson transform for $G/K$ is an isomorphism between the space of hyperfunction sections of a line bundle over the minimal boundary of $G/K$ and the space $C^{\infty}(G/K)_{\chi}$ of joint eigenfunctions of $\Diff(G/K)$ with eigencharacter $\chi$.
Then Schmid's interpretation and extension of the Helgason conjecture is
\begin{equation} \label{min is omega}
V_{\rm min}
= V^\omega
\end{equation}
as topological vector spaces, for all Harish-Chandra modules $V$. The equality (\ref{min is omega}) was stated in \cite[Theorem on p. 317]{S} and \cite[Theorem 2.12]{KashiwaraSchmid}.

\par The objective of this work is to understand the equality \eqref{min is omega} quantitatively. For that
let $G=KAN$ be an Iwasawa decomposition of $G$ and $G=KAK$ the associated Cartan
decomposition. Let $\|\cdot\|$ be a Cartan-Killing norm on $\gf$, and define balls
$A_R\subset A$ for any $R>0$ by $A_R =\exp(\af_R)$ and
$\af_R:=\{ X\in\af\mid \|X\|\leq R\}$.  This gives us a family of balls
$B_R:= K A_R K\subset G$, and we write $C_R^\infty(G)\subset C_c^\infty(G)$ for the subspace of functions with support in $B_R$. We define
$$V^{\rm min}_R:= C_R^\infty(G)*V$$
and endow it with the quotient topology.
Note that
\begin{equation} \label{filter one}
V_{\rm min}
= \varinjlim_{R\to \infty} V^{\rm min}_R
\end{equation}
as inductive limit of Fr\'echet spaces.

\par Next we consider the analytic filtration
of $V^\omega$. We recall that a vector $v\in V^\infty$ is analytic if and only if it is
$K$-analytic, i.e.~the restricted orbit map
$$f_v : K \to V^\infty, \ \ k\mapsto k\cdot v$$
is analytic (see Lemma \ref{lemma K-omega}).  Now for any
$r>0$ we define a $K$-bi-invariant domain of  $K_\C$ by $K_\C(r):= K \exp(i\kf_r)$, where
$$
\kf_r
=\{X\in \kf\mid \|X\|<r\}.
$$
We define $V_{r}^\omega\subset V^\omega$ to be the subspace of those $v$ for which $f_v$ extends holomorphically to $K_\C(r)$ and endow it with the topology of uniform convergence on compacta in $K_{\C}(r)$. We then obtain the analytic filtration  as inductive limit of Fr{\'e}chet spaces
\begin{equation}\label{filter two}
V^\omega
=\varinjlim_{r\to 0} V_r^\omega\, .
\end{equation}
By  (\ref{min is omega}) and the Grothendieck factorization theorem \cite[Ch.~4, Sect.~5, Th.~1]{Gr} (see also \cite[Corollary 24.35]{MV}) the two filtrations \eqref{filter one}
and \eqref{filter two} are continuously sandwiched into each other, i.e.:
\medbreak
\begin{enumerate}[]
\item{\em Geometric inclusion:} For all $R>0$ there exists $r=r(R)>0$ with $V^{\rm min}_R\subset V_{r}^\omega$.
\vspace{6pt}
\item{\em Analytic inclusion:} For all $r>0$ there exists $R=R(r)>0$ with $V_r^\omega\subset V^{\rm min}_{R}$.
\end{enumerate}
\vspace{6pt}
Observe that the equality (\ref{min is omega}) is a consequence of the analytic inclusion.
\medbreak
By a {\em Paley-Wiener type theorem for a Harish-Chandra module $V$} we understand the existence of the geometric and analytic inclusions together with bounds on the numbers $r(R)$ and $R(r)$. In this article we prove such a theorem.

To explain the terminology, we consider the following algebraic type of  Fourier transform
$$
\F=\bigoplus_{V\in\mathcal{HC}}\F_{V}: C_{c}^{\infty}(G)\to \bigoplus_{V\in \mathcal{HC}}\Hom_{(\gf,K)}(V,V^{\omega}),
$$
that is given by
$$
\F_{V}\phi(v)=\phi* v\qquad(V\in\mathcal{HC}, v\in V ).
$$
Here $\mathcal{HC} $ is the category of Harish-Chandra modules. A complete Paley-Wiener theorem would be a description of the image under $\F$ of the filtration of $C_{c}^{\infty}(G)$. A step towards that is the localized version, i.e.~ for a fixed $V\in\mathcal{HC}$ a description for the image under $\F_{V}$ of the filtration of $C_{c}^{\infty}(G)$ in terms of the filtration on $\Hom_{(\gf,K)}(V,V^{\omega})$ induced from $V^{\omega}$. Optimal estimates of $r(R)$ and $R(r)$ determining the geometric and analytic inclusions are an interesting open problem, even for groups of rank $1$.

\subsection{Geometric inclusion} What we termed  geometric inclusion
has a straightforward relation to a problem concerning the complex geometry of the
$G$-invariant crown domain $\Xi\subset Z_\C=G_\C/K_\C$ of the attached Riemannian symmetric space
$Z=G/K$. The crown domain $\Xi$ was first defined in \cite{AG} as in \eqref{crown1}-\eqref{crown2}  below and characterized
as the largest $G$-domain $Z\subset \Xi \subset Z_\C$ on which $G$ acts properly. If $z_0=K_\C$ is the standard base point, then the crown domain
can alternatively be defined as the connected component of the intersection
$$\bigcap_{g\in G} g N_\C A_\C\cdot z_0 = \bigcap_{k\in K} k N_\C A_\C\cdot z_0$$
which contains $z_0$. The latter can also be rephrased by $\Xi\subset Z_\C$ being the maximal $G$-invariant domain
containing $Z$ such that for every $K$-spherical principal series representation $V=V_\lambda$ with $\lambda\in\af_{\C}^{*}$ and non-zero $K$-spherical vector
$v_K=v_{K,\lambda}$ the orbit map
$$f_\lambda: G/K\to V_\lambda^\infty, \ \ gK\mapsto \pi_\lambda(g) v_{K,\lambda}$$
extends as a holomorphic map to $\Xi\to V_\lambda^\infty$. (See \cite{K-S1} and \cite{KS} for the fact that every
$f_\lambda$ extends holomorphically to $\Xi$, and \cite[Sect. 4]{KO} for the fact that $\Xi$ is maximal with respect to this
property.)

Given $R>0$ we define an $\Ad(K)$-invariant open subset in $\kf$ by
$$ \kf(R):= \{ X \in \kf\mid \exp(iX) B_R\cdot z_0\subset \Xi\}_0\ ,$$
with the subscript indicating the connected component which contains $0\in \kf$.

\begin{prop}
The following assertions hold.
\begin{enumerate}[(i)]
\item\label{Prop geometric inclusion 1} For any $r>0$ with $\kf_r\subset \kf(R)$ we have a continuous embedding $V^{\rm min}_R\subset
V_r^\omega$.
\item\label{Prop geometric inclusion 2} There exist constants $c, C>0$ so that
\begin{equation}\label{eq bound on k(R)}
\kf_r \subset \kf(R)\quad\text{if}\quad  r < C e^{-cR}.
\end{equation}
\end{enumerate}
\end{prop}

Assertion (\ref{Prop geometric inclusion 1}) is Proposition \ref{prop geom}; assertion (\ref{Prop geometric inclusion 2}) is Proposition \ref{prop suboptimal r}.

It is an interesting problem to determine $\kf(R)$ explicitly, and we do so for two examples in Appendix \ref{Appendix A}. The results
in the appendix suggest that the bound (\ref{eq bound on k(R)}) is sharp modulo the constants $c, C>0$.

\subsection{Analytic inclusion}
We now address the  more interesting and much more difficult part, namely the analytic inclusion, i.e.~to find for given $r>0$ an
$R=R(r)>0$ such that  $V_r^\omega\subset V_{R(r)}^{\rm min}$.
The main theorem of this paper is (see Theorem \ref{main theorem} with Remark \ref{rmk intro proof}):

\begin{theorem} \label{intro theorem} Let $G$ be a real reductive group and $V$ be a Harish-Chandra module. Then there exist
constants $c>0$ and $R_0>0$, only depending on $G$, with the following property:
Given $r>0$, then for all $R>R_0$ satisfying
$$ \frac{(\log R)^2}{R^2} < c r $$
one has a continuous embedding
$$V_{r}^\omega\subset V_R^{\rm min}\, .$$
\end{theorem}

As a corollary of this theorem we obtain Schmid's identity (\ref{min is omega}), and
we can view  Theorem \ref{intro theorem} as a new quantitative version of it. In Appendix \ref{Appendix B} we give a short derivation of the Helgason conjecture from (\ref{min is omega}). Finally, in Appendix \ref{Appendix C} we give
an application of our quantitative version to the  factorization of analytic eigenfunctions in terms of the Harish-Chandra
spherical function.

Let us now  explain the idea of the proof. Standard techniques reduce matters quickly to the case
when $V=V_\lambda$ is a principal series for which the $K$-spherical vector is cyclic (see Lemma \ref{lemma quotient}).
Our approach is based on the Paley-Wiener theorem of Helgason for the Fourier transform on $G/K$.
Let us briefly recall the statement.
Let $\PW(\af_\C^*, C^\infty(K/M))_R$
be the $C^\infty(K/M)$-valued Paley-Wiener space of holomorphic functions on the complexification $\af_\C^*$ of the Euclidean space $\af$ with growth rate $R$, see  \eqref{eq Paley-Wiener estimate} for
the formal definition. We realize $V_\lambda$ in the compact picture, where $V_\lambda^\infty = C^\infty(K/M)$
as $K$-modules, and  denote by $v_{K,\lambda}=\1_{K/M}$ the constant indicator function of $K/M$. It is then easy to see
that the spherical Fourier transform

$$ \F: C_{c}^{\infty}(G)\to \Oc(\af_\C^*,C^\infty(K/M)), \ \ f\mapsto (\lambda\mapsto \pi_\lambda(f) v_{K,\lambda})$$
satisfies
$$ \F (C_{R}^{\infty}(G))\subset \PW(\af_\C^*, C^\infty(K/M))_R\, .$$

Let $W$ be the Weyl group of $\Sigma(\gf,\af)$. For $w\in W$ we denote by
$$
J_{w,\lambda}: V_\lambda^\infty \simeq C^\infty(K/M)\to V_{w\lambda}^\infty\simeq C^\infty(K/M)
$$
the normalized (i.e.~fixing $\1_{K/M}$)  intertwining operator and recall that $\lambda \mapsto J_{w,\lambda}$ is meromorphic.
With that we obtain an action of $W$ on the space of $C^\infty(K/M)$-valued meromorphic functions,
$$ W\times \Mf(\af_\C^*, C^\infty(K/M))\to \Mf(\af_\C^*, C^\infty(K/M)), \ \ (w, f)\mapsto w \circ f\ ,$$
$$ (w\circ f)(\lambda) = J_{w, w^{-1}\lambda} f(w^{-1}\lambda)\qquad (\lambda \in \af_\C^*)\, .$$
In this framework Helgason's Paley-Wiener theorem \cite{H2} asserts that
$$ \F (C^{\infty}_R(G))= \PW_W(\af_\C^*, C^\infty(K/M))_R\ ,$$
where the subscript $W$ refers to invariant functions
for the action defined above.
However, from the geometric inclusion it follows that $\F(C_R^\infty(G))(\lambda) \subset
C_R^\infty(G)*V_\lambda \subset V_\lambda^\omega$. Thus we observe that the intertwining
relations force analyticity, i.e.~we have

$$\PW_W(\af_\C^*, C^\infty(K/M))_R = \PW_W(\af_\C^*, C^\omega(K/M))_R\ ,$$
and this observation was the motivation for our approach to the analytic inclusion.

Fix $\lambda_0\in \af_\C^*$ such that $V_{\lambda_0}$ is cyclic for the $K$-spherical vector.
We explicitly construct for any given analytic vector $v \in V_{\lambda_0}^\omega(r)$ a holomorphic function
$$f_v: \af_\C^* \to C^\infty(K/M)$$
such that its average
$$ \Ac(f_v):=\sum_{w\in W} w\circ f_v$$
lies in the Paley-Wiener space for a certain $R>0$ and such that $\Ac(f_v)(\lambda_0)=v$.
The Paley-Wiener theorem then yields that $v\in C_R^{\infty}(G)*V_{\lambda_0}$, proving the theorem.

We point out that our proof is in essence an $\Sl(2,\R)$-proof. More precisely, in Section \ref{Section construction} we provide a variety of estimates for products of $\Gamma$-functions, which lie at the core of the construction for $G=\Sl(2,\R)$. Given the framework provided by Kostant in \cite{Kos}, the general case of a reductive group $G$ is then a consequence of the one-variable estimates in Section \ref{Section construction}.

\section{Preliminaries}
Let $G$ be the real points of a connected algebraic reductive group defined over $\R$ and let $\gf$ be its Lie algebra. Subgroups of $G$ are denoted by capitals. The corresponding subalgebras are denoted by the corresponding fraktur letter. The unitary dual of a subgroup $S$ of $G$ we denote by $\hat{S}$.

 We denote by $\gf_\C=\gf\otimes_\R \C$ the complexification of $\gf$ and by $G_{\C}$ the group of complex points. We fix a Cartan involution $\theta$ and write $K$ for the maximal compact subgroup that is fixed by $\theta$.  We also write $\theta$ for the derived automorphism of $\gf$. We write $K_{\C}$ for the complexification of $K$, i.e. $K_{\C}$ is the subgroup of $G_{\C}$ consisting of the fixed points for the analytic extension of $\theta$.

The Cartan involution induces the infinitesimal Cartan decomposition
$\gf =\kf \oplus\sf$. Let $\af\subset\sf$ be a maximal abelian
subspace. Diagonalize $\gf$ under $\ad \af$ to obtain
the familiar root space decomposition
$$\gf=\af\oplus\mf\oplus \bigoplus_{\alpha\in\Sigma} \gf^\alpha\ ,$$
with $\mf=\zf_\kf(\af)$ as usual. Let $A$ be the connected subgroup of $G$ with Lie algebra $\af$ and let $M=Z_{K}(\af)$.
We fix an Iwasawa decomposition $G=KAN$ of $G$. We define the projections $\mathbf{k}:G\to K$ and $\mathbf{a}:G\to A$ by
$$
g\in \mathbf{k}(g)\mathbf{a}(g)N\qquad(g\in G).
$$
The set of restricted roots of $\af$ in $\gf$ we denote by $\Sigma$ and the positive system determined by the Iwasawa decomposition by $\Sigma^{+}$. We write $W$ for the Weyl group of $\Sigma$.

Let $\kappa$ be the Killing form on $\gf$ and let $\tilde\kappa$ be a non-degenerate $\Ad(G)$-invariant symmetric bilinear form on $\gf$ such that its restriction to $[\gf,\gf]$ coincides with the restriction of $\kappa$ and $-\tilde\kappa(\,\cdot\,,\theta\,\cdot\,)$ is positive definite. We write $\|\cdot\|$ for the corresponding norm on $\gf$.

\section{The complex crown of a Riemannian symmetric space}
The Riemannian symmetric space
$Z=G/K$ can be realized as a totally real subvariety of the Stein symmetric space
$Z_\C= G_\C/K_\C$:
$$ Z=G/K \hookrightarrow Z_\C, \ \ gK\mapsto gK_\C\, .$$
In the following we view $Z\subset Z_\C$ and write $z_0=K\in Z$ for the standard base point.

We define the subgroups $A_\C=\exp(\af_\C)$ and
$N_\C=\exp(\nf_\C)$ of $G_\C$.
We note that
$N_\C A_\C K_\C$ is a Zariski-open subset of
$G_\C$.
The maximal $G\times K_\C$-invariant domain in
$G_\C$ containing $e$ and contained in $N_\C A_\C K_\C$ is given by
\begin{equation} \label{crown1}
\tilde \Xi
= G\exp(i\Omega)K_\C\ ,
\end{equation}
where $\Omega=\{ Y\in \af\mid (\forall \alpha\in\Sigma)
\alpha(Y)<\pi/2\}$. Taking right cosets by $K_\C$, we obtain
the $G$-domain
\begin{equation}\label{crown2} \Xi:=\tilde \Xi/K_\C \subset Z_\C=G_\C/K_\C\ ,\end{equation}
commonly referred to as the {\it crown domain}. See
\cite{Gi} for the origin of the notion, \cite[Cor.~3.3]{KS} for the inclusion
$\tilde \Xi\subset  N_\C A_\C K_\C$  and \cite[Th.~4.3]{KO} for the maximality.

We recall that $\Xi$ is a contractible space. To be more precise, let $\hat\Omega=\Ad(K)\Omega$ and note that $\hat\Omega$ is an open convex subset of $\sf$. As a consequence of the Kostant convexity theorem it satisfies $\hat\Omega\cap\af=\Omega$ and $p_{\af}\hat\Omega=\Omega$, where $p_{\af}$ is the orthogonal projection $\sf\to\af$. The fiber map
$$
G\times_{K}\hat\Omega\to\Xi; \quad [g,X]\mapsto g\exp(iX)\cdot K_{\C}\ ,
$$
is a diffeomorphism by \cite[Prop.~4, 5 and 7]{AG}.  Since $G/K\simeq\sf$ and $\hat\Omega$ are both contractible, also $\Xi$ is contractible. In particular, $\Xi$ is simply connected.

\par We denote by ${\bf a}: G \to A$ the middle projection of the Iwasawa decomposition $G=KAN$ and note that
${\bf a}$ extends holomorphically to
$$
\tilde \Xi^{-1}
:=\{g^{-1}:g\in\tilde\Xi\}\ .
$$
Here the simply connectedness of $\Xi$ plays a role  to achieve ${\bf a}: \tilde \Xi^{-1}\to A_\C$ uniquely: A priori ${\bf a}$ is only defined as a map to  $A_\C/T_2$, where $T_2:= A_\C \cap K_\C$ is the $2$-torsion subgroup of group $A_\C$.
We denote the extension by the same symbol.

 Likewise one defines  $\mathbf{k}: G \to K$, which extends holomorphically to $\tilde \Xi^{-1}$ as well.

\par For $R>0$ we define a ball in $A$ by
$$A_R:=\{ \exp(Y)\mid Y\in \af, \ \|Y\|\leq R\}\, .$$
Related to that we define the ball $B_R\subset G$ by
$B_R=K A_R K$.
We consider the following subset of $\kf$:
\begin{eqnarray}\nonumber \kf(R)&:=\{ Y\in \kf\mid \exp(iY)K A_R \subset \tilde\Xi\}_0\\
\label{def k(R) 2} & =\{ Y\in \kf\mid \exp(iY)B_R\subset \tilde\Xi\}_0\, .
\end{eqnarray}

Note that $\kf(R)$ is open, because $KA_R\subset G$ is compact. Moreover, $\kf(R)$ is
$\Ad(K)$-invariant. Hence it is uniquely determined
by its intersection with a Cartan subalgebra $\tf$ of $\kf$, i.e.~$\kf(R)$  is determined by
$$
\tf(R):=\tf\cap \kf(R)\, .
$$
Actually, it is sufficient to consider the intersection with a closed chamber of $\tf$, say $\tf^+$:
$$\tf(R)^+:=\tf^+\cap \kf(R)\, .$$

For $r>0$ let
$\kf_r:=\{X\in \kf\mid \|X\|<r\}$ and define the domains in $K_{\C}$
$$
K_{\C}(r)
:=K\exp(i\kf_{r})\ .
$$
Note that $K_{\C}(r)$ is $K$-biinvariant, as $\kf_{r}$ is $\Ad(K)$-invariant. Note further that $K_{\C}(r)=\big(K_{\C}(r)\big)^{-1}$, since $\kf_{r}=-\kf_{r}$.
\par In general it is an interesting problem to determine $\kf(R)$ explicitly. We do this in Appendix \ref{Appendix A} for two cases, namely $\gf=\so(1,n)$ and $\gf=\su(1,1)$, the latter being treated in a way so that the
generalization to Hermitian symmetric spaces becomes apparent.
\par As a precise description of $\kf(R)$ may be difficult to obtain in general, one could instead determine the best possible
$r=r(R)>0$ with $\kf_{r}\subset \kf(R)$.
The following proposition gives a first bound which, given the results in Appendix \ref{Appendix A}, appears to be sharp
up to constants.

\begin{prop} \label{prop suboptimal r} There exist constants $C, c>0$ such that for all $r, R>0$  one has
$$
\kf_r
\subset \kf(R)\qquad (r < C e^{-cR}).
$$
\end{prop}

\begin{proof} Let $G=\GL(n,\R)$. We consider the standard Iwasawa decomposition of $G$,
i.e.~$K=\OO(n,\R)$, $A=\diag(n, \R_{>0})$ and $N$ is the group of unipotent upper triangular matrices. It
suffices to consider this case, as any real reductive group can be embedded into $G=\GL(n,\R)$ with compatible
Iwasawa decompositions. Here we remark that the possible incompatibility of the Cartan-Killing norms is taken care of by the presence of the constants $C$ and $c$.

We recall that
$$ Z= G/ K \to \Sym(n,\R)^+, \ \ gK\mapsto  g g^t\ ,$$
identifies $Z$ with the positive definite symmetric matrices.  In this matrix picture
$Z_\C$ is identified with $\Sym(n,\C)_{\det\neq 0}$, the invertible symmetric matrices.  In case $n\geq 3$, the crown domain
$\Xi\subset \Sym(n,\C)_{\det\neq 0}$ is not explicitly known. However, $\Xi$ contains the so-called
square root domain
\begin{equation} \label{sq domain} \Xi^{\frac{1}{2}}= \Sym(n,\R)^+ + i \Sym(n,\R) \subset \Sym(n,\C)_{\det\neq 0}\ ,\end{equation}
see \cite[Sect.~8]{K-S1}.
Let $m=\left[\frac{n}{2}\right]$ and define for $x \in \R^m$
$$D(x) = \begin{pmatrix} D_1(x) & & \\ & \ddots& \\ & & D_m(x)\end{pmatrix} \ ,$$
with
$$D_j(x)= \begin{pmatrix} \cos x_j  & -\sin x_j \\ \sin x_j & \cos x_j\end{pmatrix}\,.$$
In case $n$ is even we have $D(x) \in \OO(n,\R)$, and in case of $n$ odd we view
$D(x) \in \OO(n,\R)$ by means of the embedding
$$D(x)\mapsto \begin{pmatrix} D(x) & \\ & 1\end{pmatrix}\, .$$
Our choice of maximal torus $T\subset K$ then is $T=\{ D(x) \mid x \in \R^m\}$.
\par Let now $R>0$ and $Y \in \Sym(n,\R)^+$ with $\Spec (Y) \subset [e^{-R}, e^{R}]$.
We then seek an $r>0$ such that for all $x\in \R^m$ with $\|x\|<r$  and $Y$ as above we have $D(ix) Y D(ix)^t \in \Xi^{\frac{1}{2}}$.
If we decompose $D(ix) = U(x) + i V(x)$ into real and imaginary parts, this amounts to
$$ U(x) Y U(x)  - V(x) Y V(x)^t \in \Sym(n,\R)^+\ ,$$
by \eqref{sq domain}.  With
$$S(x) = \begin{pmatrix} \begin{pmatrix} 0 & -\tanh x_1 \\  \tanh x_1&0 \end{pmatrix} & & \\
& \ddots & \\
& & \begin{pmatrix} 0 & -\tanh x_m \\  \tanh x_m&0 \end{pmatrix} \end{pmatrix} $$
we can rewrite this as
\begin{equation}\label{YS cond} Y - S(x) Y S(x)^t \in \Sym(n,\R)^+\,. \end{equation}
Now  note that
$$ \| S(x) Y S(x)^t\|_{\rm op} \leq [\tanh r]^2 \|Y\|_{\rm op} \leq [\tanh r]^2 e^{R}\, .$$
 On the other hand, the smallest eigenvalue of $Y$ is at least $e^{-R}$. Hence \eqref{YS cond} is satisfied,
 provided $[\tanh r]^2 e^{2R} < 1$. As $\tanh r\leq r $, \eqref{YS cond}  is implied by $r^2 < e^{-2R}$, and the assertion of the proposition follows.
 \end{proof}

\section{Generalities on the analytic filtration}\label{Section Reduction}

\subsection{Filtration by holomorphic extension}
Let $V$ be a Harish-Chandra module.
The analytic vectors $V^\omega$ of $V$ are defined as the analytic vectors in $V^\infty$, i.e.~$V^\omega:= (V^\infty)^\omega$.
In the following we provide various standard descriptions of $V^\omega$.

The first one is in terms of holomorphic extensions.
For $r>0$ we define
$$
V_{r}^{\omega}
:=\{v\in V^{\infty} \mid K\ni k\mapsto k\cdot v\in V^{\infty} \text{ extends holomorphically to }K_{\C}(r)\}\, .
$$

\begin{lemma}\label{lemma K-omega}
For any Harish-Chandra module $V$ every $K$-analytic vector is analytic.  Moreover,
$$
V^\omega
=\varinjlim_{r\to 0} V_r^\omega
$$
as topological vector spaces.

\end{lemma}

\begin{proof} From the definition it is easily checked that $\varinjlim_{r\to 0} V_r^\omega$ describes the space of analytic vectors for the representation on $V^\infty$ restricted to $K$ with the topology of uniform convergence on $K$.

We recall the notion of $\Delta$-analytic vectors from \cite[Sect. 5]{GKL} and that the space of $\Delta$-analytic vectors coincides with the space of analytic vectors for any F-representa\-tion of a Lie group, and in particular for $V^{\infty}$.
Let $\Cc$ be the Casimir element and let $\Delta_K$ and $\Delta_G$ be the standard Laplace elements
in $\U(\kf)$ and $\U(\gf)$, respectively. Then $\Delta_G = \Cc +2\Delta_K$. As
$\Delta_G$ differs from $2\Delta_K$ by $\Cc$, which acts finitely on $V$, it follows that any $\Delta_K$-analytic vector is $\Delta_G$-analytic, and vice versa. This proves the first assertion.

The identity map from the space of $G$-analytic vectors to the space of $K$-analytic vectors is continuous. The second assertion now follows from the open mapping theorem (see \cite[Theorem 24.30 and Remark 24.36]{MV}).
\end{proof}

\subsection{Filtration by $K$-type decay}
The next description of analytic vectors is by exponential decay of $K$-types.
A norm $p$ on $V$ is called $G$-continuous provided that the completion $V_{p}$ of the normed space $(V, p)$ gives rise
to a Banach-representation of $G$. We choose a $G$-continuous norm $p$ on $V$. Let $V^{\infty}$ be the up to isomorphism
unique smooth completion of $V$ with moderate growth, see \cite{C}, \cite[Ch.~11]{W2} or
\cite{BK}.
We write a vector $v\in V^{\infty}$ as a convergent sum
 $$
 v
 =\sum_{\tau\in\hat{K}}v_{\tau}\ ,
 $$
where $v_{\tau}$ is contained in the $K$-isotypical component $V[\tau]$ of $V$. For any $\tau\in \hat K$ we denote by $|\tau|$ the norm of the highest weight of $\tau$.

For $r>0$ let us define
$$
V^{\omega}(r)
:=\{v\in V^{\infty}\mid (\forall 0<r'<r)\  \sum_{\tau\in\hat{K}}e^{r'|\tau|}p(v_{\tau})<\infty\}
$$
and endow it with the Fr{\'e}chet topology induced by the seminorms
$$
v\mapsto \sum_{\tau\in\hat{K}}e^{r'|\tau|}p(v_{\tau})
\qquad(0<r'<r)\,.
$$

The space $V^{\omega}(r)$ is independent of the choice of the $G$-continuous norm $p$, as all these norms are polynomially comparable on the $K$-types, i.e.~given two $G$ continuous norms
$p$ and $q$ on $V$ there exists a constant $C>0$, so that $p|_{V[\tau]}\leq C ( 1 + |\tau|)^C q|_{V[\tau]}$ for all $\tau\in\hat{K}$, see
\cite[Th.~1.1]{BK}.

\begin{lemma}\label{lemma K-analytic}For every Harish-Chandra module $V$  we have
$$
V^{\omega}_r = V^{\omega}(r) \qquad (r>0)
$$
as topological vector spaces.
\end{lemma}

\begin{proof}
Let $r>0$. We first prove the inclusion $ V^{\omega}(r)\subset V^{\omega}_r$. For this let $v \in V^\omega(r)$ and $0<r'<r$. Recall that $K_\C(r')= K \exp(i\tf_{r'}) K$ with
$\tf_{r'}=\{ X \in \tf\mid \|X\|<r'\}$.  Since the space $V^{\omega}(r)$ is independent of the choice of the $G$-continuous norm
$p$, we may assume that $p$ is Hermitian and $K$-unitary.
Any element $t\in\exp(i\tf_{r'}) $ acts semisimply on $V[\tau]$ with eigenvalues bounded by $e^{r'|\tau|}$.
As $p$ is $K$-unitary, it follows that
\begin{equation} \label{ineq1} \sup_{k\in K_\C(r')} p (k \cdot v _\tau)\leq e^{r' |\tau|} p(v_\tau)\qquad (\tau \in \hat K)\, .\end{equation}
We recall that $V_p$ is the Hilbert completion of $V$ with respect to $p$ and that $V_{p}$ is a Hilbert representation of $G$. Because $v \in V^\omega(r)$, inequality \eqref{ineq1}
and the fact that $\dim V[\tau]$ is polynomially bounded in $|\tau|$
imply that the orbit map
$$ f_v: K \to V_p, \ \ k\mapsto k\cdot v\ ,$$
extends holomorphically to $K_\C(r')$. Since this holds for all $r'<r$, the function $f_{v}$ in fact extends holomorphically to $K_{\C}(r)$.
The image of $f_v$ is not only in $V_p$, but in the $K$-smooth vectors of $V_p$. Since the Fr{\'e}chet spaces of  $K$-smooth and $G$-smooth vectors in $V_{p}$ are identical (see \cite[Corollary 3.10]{BK}), we obtain that $f_v$ is a holomorphic map with values in $V_p^\infty = V^\infty$. Thus
we have shown that $ V^\omega(r) \subset V_r^\omega$.
The embedding is continuous in view of (\ref{ineq1}).

\par For the converse inclusion $V^{\omega}_r \subset V^{\omega}(r)$, we note that
for an irreducible Harish-Chandra module $V$ the representation
$V^{\infty}$ can be embedded into the space of smooth vectors of a minimal principal series module $V_{\sigma,\lambda}$. The latter can be realized as the space of smooth functions $f:G\to V_{\sigma}$ satisfying
$$
f(gman)
=a^{-i\lambda -\rho}\sigma(m)^{-1}f(g)\qquad
(g\in G, man\in MAN)\, .
$$
Note that $V_{\sigma,\lambda}^\infty$ is naturally
a $G$-module, with $G$ acting on $V_{\sigma,\lambda}^\infty$ by
left displacements in the arguments, in symbols
$\pi_{\sigma,\lambda}(g)(f)=f(g^{-1} \cdot )$.
We write $\Hc$ for $C^{\infty}(K)$
equipped with the $G$-representation $\pi_{\lambda}$ given by
$$
\big(\pi_{\lambda}(g)f\big)(k)=\mathbf{a}(g^{-1}k)^{-i\lambda-\rho}f(\mathbf{k}\big(g^{-1}k)\big)
\qquad\big(f\in \Hc,g\in G, k\in K\big)\ .
$$
We may embed $V^{\infty}$ equivariantly into $\Hc\otimes V_{\sigma}$.
It therefore suffices to prove that $\Hc_{r}^{\omega}\subset\Hc^{\omega}(r)$.

We let $p$ be the $L^2$-norm on $\Hc$, which is $G$-continuous.
Note that $K$ acts also from the right on smooth functions on $K$, and therefore $\Hc$ carries a representation of $K\times K$. From now on we consider $\Hc$ as a $K\times K$ module.
For $0<r'<r$ we  define a $K\times K$-invariant Hermitian norm
on $\Hc_{r}^{\omega}$ by
$$
q_{r'}(v)
:= \left[\int_{K_\C(r')} |v(k)|^2 \ d\mu(k)\right]^{\frac{1}{2}}
\qquad (v\in \Hc_{r}^{\omega})\ .
$$
Here $d\mu$ is the measure on $K_{\C}$ which in the polar decomposition $K_{\C}=K\exp(i\tf^{+})K$ is given by
$$
d\mu(k_{1}\exp(it)k_{2})
=dk_{1}\,dt\,dk_{2}\ ,
$$
with $dk_{1,2}$ the Haar measure on $K$ and $dt$ the Lebesgue measure on $\tf^{+}$. For $\tau\in\hat{K}$ we define $\Hc[\tau]$ to be the $\tau\otimes \tau^{\vee}$-isotypical component of $\Hc$ and denote the restriction of $p$ and $q_{r'}$ to $\Hc[\tau]$ by $p_{\tau}$ and $q_{r',\tau}$, respectively.
Since $\Hc[\tau]$ is $K\times K$-irreducible, there exists a constant
$c_{r',\tau}>0$, so that $q_{r',\tau} = c_{r',\tau} \cdot p_\tau$.

We will estimate the constant from below by estimating $q_{r',\tau}(v)$ for a matrix coefficient
$$
v=m_{w_{1},w_{2}}:k\mapsto \langle w_{1},\tau(k)w_{2}\rangle\ ,
$$
 where $w_{1},w_{2}\in V_{\tau}$. Using the Schur-Weyl orthogonality relations we obtain
$$
q_{r',\tau}(v)^{2}
=\frac{\|w_{1}\|^{2}}{\dim \tau}\int_{K}\int_{\tf_{r' }^{+}}\|\tau\big(\exp(it)k\big)w_{2}\|^{2}\,dk\,dt\ .
$$
Next we pick an orthonormal basis of weight vectors $v_{1},\dots,v_{n}\in V_{\tau}$ and expand the integrand. We thus obtain that the right-hand side is equal to
$$
\frac{\|w_{1}\|^{2}}{\dim \tau}\sum_{j=1}^{n}\int_{K}\int_{\tf_{r'}^{+}}|\langle\tau(k)w_{2},\tau(\exp(it))v_{j}\rangle|^{2}\,dk\,dt\ .
$$
Now we apply Schur-Weyl once more. This yields
$$
\frac{\|w_{1}\|^{2}\|w_{2}\|^{2}}{\dim(\tau)^{2}}\sum_{j=1}^{n}\int_{\tf_{r'}^{+}}\|\tau(\exp(it))v_{j}\|^{2}\,dt\ .
$$
Again by Schur-Weyl we note
$$
\frac{\|w_{1}\|^{2}\|w_{2}\|^{2}}{\dim(\tau)}
=p_{\tau}(v)^{2}\ .
$$

Let $\mu_{\tau}$ be the highest weight of $\tau$, and assume that $v_{1}$ is a highest weight vector with weight $\mu_{\tau}$. Then for all $r''<r'$ there exists a constant $c$, independent of $\tau$, so that
$$
q_{r',\tau}(v)^{2}
\geq \frac{p_{\tau}(v)^{2}}{\dim(\tau)}\int_{\tf_{r'}^{+}}e^{2\mu_{\tau}(it)}\,dt
\geq c^{2} e^{2\|\mu_{\tau}\|r''}p_{\tau}(v)^{2}\ .
$$
As $\|\mu_{\tau}\|=|\tau|$, we conclude that for every $r''<r'$ there exists a constant $c_{r''}>0$, so that
$$
c_{r',\tau}
\geq c_{r''} e^{r'' |\tau|}
\qquad(\tau\in\hat{K})\ .
$$

If $v=\sum_{\tau\in\hat{K}}v_{\tau} \in \Hc_r^\omega$, then for all $0<r''<r'<r$
$$
\sum_{\tau\in\hat{K}}e^{|\tau| r''}p_{\tau}(v_{\tau})
\leq\frac{1}{c_{r''}}\sum_{\tau\in\hat{K}}q_{r',\tau}(v_{\tau})
=\frac{1}{c_{r''}}q_{r'}(v)<\infty\ .
$$
It follows that $v \in \Hc_r^\omega$ implies $v \in \Hc^\omega(r)$.
Moreover, the embedding is continuous.
\end{proof}

\subsection{Reduction to spherical principal series}\label{Subsection reduction}
It is our intention to show for a given Harish-Chandra module $V$ and $r>0$
that there is a continuous embedding
$$
V^\omega(r)
\subset V_R^{\rm min}
$$
for some $R=R(r)>0$.
For $\lambda\in\af_{\C}^{*}$ we write $ V_\lambda$ for the spherical principal series representation $\Ind_P^G (\C_{i\lambda})$.
We will first reduce the problem to the case in which $V=V_\lambda$ for some $\lambda\in\af_{\C}^{*}$.
\par Every irreducible Harish-Chandra module $V$ is a quotient
\begin{equation}\label{eq V quotient of V_lambda times F}
V_\lambda\otimes F
\twoheadrightarrow V
\end{equation}
for some spherical principal series $V_\lambda$ and finite dimensional
representation $F$ of $G$, see \cite[Sect.~2]{W}. We first recall how this arises.  By the Casselman embedding theorem every
irreducible Harish-Chandra module $V$ is a quotient of some minimal principal series module $V_{\sigma, \lambda} = \Ind_P^G (V_{\sigma}\otimes\C_{i\lambda})$ with $(\sigma,V_{\sigma})\in\hat M$ and $\lambda\in\af_{\C}^{*}$, i.e.~
$$
V_{\sigma, \lambda} \twoheadrightarrow V\, .
$$

Now, by op.~cit.~the $M$-representation $(\sigma, V_\sigma)$ can be realized  as the quotient
$F/ \nf F$ of a finite dimensional module $F$ of $G$, i.e.~$V_\sigma= F/ \nf F$.  By the Mackey isomorphism
we have
$$ V_\lambda \otimes F = \Ind_P^G (\C_{i\lambda}) \otimes F \simeq \Ind_P^G (\C_{i\lambda}\otimes F|_{P})\ .$$
Hence the $P$-morphism $\C_{i\lambda}\otimes F|_{P} \to \C_{i\lambda} \otimes F/  \nf F  \simeq \C_{i\lambda}\otimes V_\sigma$
gives rise to the chain of quotients
$$ V_\lambda \otimes F \twoheadrightarrow V_{\sigma, \lambda} \twoheadrightarrow V\, .$$
This proves (\ref{eq V quotient of V_lambda times F}).

\begin{lemma}\label{lemma reduction 1}
Let $\lambda\in\af_{\C}^{*}$ and $F$ a finite dimensional representation of $G$. The following assertions hold.
\begin{enumerate}[(i)]
\item\label{lemma reduction 1 - item 1} Let $r, R>0$. If
$V_{\lambda}^\omega(r)$ embeds continuously into $C_R^\infty (G)*V_\lambda $, then also
$ (V_\lambda\otimes F)^\omega(r) $ embeds continuously into $ C_R^\infty(G)* (V_\lambda\otimes F)$.

\item\label{lemma reduction 1 - item 2} Let $V$ be a Harish-Chandra module so that there exists a quotient map $\varpi:V_{\lambda}\otimes F\to V$. Then for every $r>0$
$$
V^{\omega}(r)
=\varpi\Big( (V_\lambda\otimes F)^\omega(r)\Big)\, ,
$$
where the symbol $\varpi$ is also used for its the globalization to $(V_{\lambda}\otimes F)^{\infty}$.
 \end{enumerate}
\end{lemma}

\begin{proof} First note that $ (V_\lambda\otimes F)^\omega(r)= V_{\lambda}^\omega (r)\otimes F$, as
$F$ is in fact a $G_\C$-module.
To prove (\ref{lemma reduction 1 - item 1}), it thus suffices to show that $\big(C_{R}^{\infty}(G)*V_{\lambda}\big)\otimes F$
continuously embeds into
$C_{R}^{\infty}(G)*(V_{\lambda}\otimes F)$.
The proof for this is analogous to \cite[Lemma 9.4]{BK}.

We move on to (\ref{lemma reduction 1 - item 2}).
Let $p$ be a $K$-invariant $G$-continuous Hermitian norm on $V_{\lambda}\otimes F$. Let $q$ be the corresponding quotient norm on $V$. Then $q$ is $G$-continuous and $K$-invariant. Note that the definition of $V^{\omega}(r)$ does not depend on the choice of the $G$-continuous norm on $V$.
Assertion (\ref{lemma reduction 1 - item 2}) now follows, because $V^{\omega}(r)$ as a $K$-module is a direct summand of $(V_{\lambda}\otimes F)^{\omega}(r)$.
\end{proof}

\subsection{Kostant's condition}
We would like to be more restrictive on the parameter $\lambda$ of the quotient $V_\lambda\otimes F \twoheadrightarrow V$.

\begin{lemma}\label{lemma quotient} Every irreducible Harish-Chandra module $V$ admits a quotient $V_\lambda \otimes F \twoheadrightarrow V$,
where $F$ is an irreducible finite dimensional representation of $G$ and $\lambda\in \af_\C^*$
satisfies the Kostant condition
\begin{equation} \label{Kos}
\re (i\lambda)(\alpha^\vee)\geq 0 \qquad  (\alpha\in \Sigma^+)\, .
\end{equation}
If (\ref{Kos}) is satisfied, then  $V_\lambda = \U(\gf) v_{K,\lambda}$ is $ \U(\gf)$-cyclic for the $K$-fixed vector $v_{K,\lambda}$.
\end{lemma}

\begin{proof}
In view of \eqref{eq V quotient of V_lambda times F}, $V$ admits a quotient $V_{\lambda}\otimes F\twoheadrightarrow V$, where $F$ is a finite dimensional representation of $G$.
Let $F'$ be a $K$-spherical finite dimensional representation of  lowest weight $-\mu \in \af^*$, where $\mu$ is dominant.
Then $M$ acts trivially on the $MA$-module $F'/\nf F'\simeq \C_{-\mu}$ with $A$-weight $-\mu$.
In particular, we obtain a quotient
$$
V_{\lambda-i\mu}\otimes F' \twoheadrightarrow V_{\lambda }\, .
$$
It follows that $V$ admits a quotient $V_{\lambda-i\mu}\otimes F\otimes F'\twoheadrightarrow V$.
The first assertion now follows by taking $\mu$ sufficiently large. The last assertion is \cite[Th.~8]{Kos}.
\end{proof}

\section{The geometric inclusion}
The goal of this section is to show that $V_R^{\rm min} = C_R^\infty(G)*V$
embeds into
$V^\omega(r)$ for any $r>0$ with $\kf_r\subset \kf(R)$. This reduces to the case where $V=V_\lambda$ is the spherical principal series representation with parameter $\lambda\in\af^{*}_{\C}$. Elements in $V_\lambda^\infty$ are uniquely determined by their restriction to $K$. This gives rise to the compact model, in which
\begin{itemize}
\item $V_\lambda^\infty= C^\infty(K/M)$,
\item $V_\lambda^\omega= C^\omega(K/M)$,
\item $V_\lambda= \C [K_\C/M_\C] $
\end{itemize}
 as $K$-modules. The main result of this section is the following.

\begin{prop} \label{prop geom} Let $V$ be a Harish-Chandra module, and $R>0$. Let $r$ be such that $\kf_r\subset \kf(R)$. Then we have the continuous embedding
$$V_R^{\rm min} \subset V^\omega(r)\, .$$
\end{prop}

\begin{proof} We first reduce to the case where $V=V_\lambda$ is a spherical principal series.
We recall from Section \ref{Subsection reduction} that $V$ is a quotient of some $V_\lambda \otimes F$, with $F$ a finite dimensional representation. Now all matrix coefficients of $F$ extend holomorphically to
$G_\C$, and this completes the reduction to $V=V_\lambda$.
\par  We work in the compact model
of $V_\lambda$. Let
$v=\pi(f)w$ for some $w\in V$ and
$f\in C_R^\infty(G)$. Then we note that for $k\in K$
\begin{equation}\label{eq Formula v(k)}
v(k)=\pi(f)(w)(k)=\int_{B_R} f(g) w(g^{-1}k) \ dg \, .
\end{equation}
Observe that $w(g^{-1}k)= w ({\bf k}(g^{-1}k)) {\bf a}(g^{-1} k)^{-i\lambda - \rho}$. As $w\in \C[K_\C/M_\C]$, $w$ is a holomorphic function  on $K_\C/ M_\C$.  Thus with
$B_R K_\C(r) \subset \tilde\Xi^{-1} \subset K_\C A_\C N_\C$, we conclude that ${\bf a}$ and ${\bf k}$ are defined on $B_R K_\C(r)$ and
holomorphic. Thus  $v$ extends to the holomorphic function on $K_\C(r)$ given by (\ref{eq Formula v(k)}). This shows the continuous embedding for this case.
\end{proof}

\section{Preliminaries on the analytic inclusion}

\subsection{$K$-type expansion of functions on $K/M$}
In the following we view functions on $K/M$ as right $M$-invariant functions on $K$. For any $\tau\in \hat K$
we fix a model (finite dimensional) Hilbert space $V_\tau$. For $\tau\in \hat K$ we write $\tau^\vee$ for the dual representation.
We then obtain for each $\tau\in \hat K$ a $K\times K$-equivariant
realization of $V_\tau\otimes V_{\tau^\vee}$ as polynomial functions on $K$:
$$ V_\tau \otimes V_{\tau^{\vee}}  \to \C[K_{\C}], \ \ v\otimes v^\vee \mapsto
m_{v, v^\vee}; \ m_{v, v^\vee}(k):= v^\vee(k^{-1} v)\ ,$$
where $K\times K$ acts on $\C[K_{\C}]$ by the left-right regular representation.
We arrive at the $K\times K$-isomorphism of $K\times K$-modules
$$ \C[K_{\C}]=\bigoplus_{\tau \in \hat K} V_\tau \otimes V_{\tau^\vee}\ ,$$
and taking right $M$-invariants at the $K$-isomorphism of $K$-modules
$$\C[K_{\C}/M_{\C}]= \bigoplus_{\tau \in \hat K_M} V_\tau\otimes V_{\tau^\vee}^M\ ,$$
where $\hat K_M\subset \hat K$ is the $M$-spherical part of $\hat K$.
Fix $\tau$ and identify $V_{\tau^\vee}\simeq V_\tau^*$. In particular, the unitary norm on $V_\tau$ induces the unitary
dual norm on $V_{\tau^\vee}$ and we write $\|\cdot\|_\tau$ for the Hilbert-Schmidt norm on $V_\tau\otimes V_{\tau^\vee}$.
We recall that $\|\cdot\|_\tau$ is independent of the particular unitary norm on $V_\tau$ (which is unique up to positive scalar by Schur's Lemma) and is thus intrinsically defined.
Any function on $f\in \C[K_{\C}]$ we now expand into $K$-types $f=\sum_{\tau\in \hat K} f_\tau$ with $f_\tau \in V_\tau\otimes
V_{\tau^\vee}$.  With that we record the well known Fourier characterizations of
$C^\infty(K)$ and $C^\omega(K)$ as
\begin{equation*} 
C^\infty(K)=\{ f = \sum_{\tau\in \hat K} f_\tau\mid (\forall N\in \N)\  \sum_{\tau\in \hat K} ( 1+|\tau|)^N \|f_\tau\|_\tau<\infty\}
\end{equation*}
and
\begin{equation*} 
C^\omega(K)=\{ f=\sum_{\tau \in \hat K}f_\tau  \mid(\exists r>0)\ \sum_{\tau\in \hat K} e^{r |\tau|} \|f_\tau\|_\tau<\infty\}\, .
\end{equation*}
Taking right $M$-invariants, we obtain corresponding Fourier characterizations of $C^\infty(K/M)$ and
$C^\omega(K/M)$.

\subsection{The Helgason Paley-Wiener Theorem}

We begin with a short review of the Fourier transform on $Z=G/K$ and recollect some notation.
For $\lambda\in \af_\C^*$ we denote by $V_\lambda$ the Harish-Chandra module of the $K$-spherical principal
series with parameter $\lambda$ as defined before.
Recall  that $V_\lambda^\infty = C^\infty(K/M)$ as $K$-module.  We denote by $v_{K,\lambda}={\bf 1}_{K/M}\in V_\lambda$
the constant function.

\par For every $R>0$ we let $\PW(\af_\C^*, C^\infty(K/M))_R$ be the space of holomorphic functions
$f: \af_\C^* \to C^\infty(K/M)$, so that for every continuous semi-norm $q$ on $C^\infty(K/M)$ and $N\in \N$  one has
\begin{equation}\label{eq Paley-Wiener estimate}
\sup_{\lambda\in \af_\C^*}  q (f(\lambda))  (1 +\|\lambda\|)^N e^{-R \|\im \lambda\|} <\infty\, .
\end{equation}
Further we denote
$$\PW(\af_\C^*, C^\infty(K/M))= \bigcup_{R>0} \PW(\af_\C^*, C^\infty(K/M))_R$$
and refer to it as the Paley-Wiener space on $\af_\C^*$ with values in $C^\infty(K/M)$.

The Fourier transform on $Z$ is then defined  by
\begin{align*}
&\F : C_c^\infty(Z) \to \PW(\af_\C^*, C^\infty(K/M))\ ,\\
\label{eq def F}&f\mapsto \F(f); \ \F(f)(\lambda):= \pi_\lambda(f)v_{K,\lambda}\, .
\end{align*}
Note that
$$
\F(f)(\lambda)(kM)
= \int_Z f(gK) {\bf a} (g^{-1} k)^{-i\lambda -\rho} \ d(gK)\qquad (k\in K)\,.
$$
It is convenient to write $\F(f)(\lambda, kM)$ for $\F(f)(\lambda)(kM)$.

In order to describe the image of $\F$, we recall the Weyl group $W$ of the restricted root system $\Sigma=\Sigma(\af, \gf)$.
Attached to  $w\in W$ there is a meromorphic family of standard intertwining operators
$$ I_{w,\lambda}:  V_\lambda^\infty \to V_{w\lambda}^\infty\, .$$
Further we recall that $I_{w,\lambda} (v_{K,\lambda}) =  {\bf c}_w(\lambda) v_{K,w\lambda}$ for a meromorphic and explicit
function ${\bf c}_w$ ($w$-partial Harish-Chandra ${\bf c}$-function, calculated by Gindikin-Karpelevic).  We define the normalized intertwining operator by  $J_{w,\lambda}:=\frac{1}{{\bf c}_w(\lambda)}  I_{w,\lambda}$. We recall that $\lambda \mapsto J_{w,\lambda}$ is meromorphic on $\af_{\C}^{*}$, and holomorphic on an open neighborhood of the cone
$$
\{\lambda\in\af_{\C}^{*}:\re\big(i\lambda(\alpha^{\vee})\big)\geq0 \text{ for all }\alpha\in \Sigma^{+}\cap w^{-1}\Sigma^{-}\}\ .
$$
It is clear from the definitions that every Fourier transform
$\phi=\F(f)$ satisfies the {\it intertwining relations}
\begin{equation} \label{I-relation}
 J_{w,\lambda}( \phi(\lambda) )
= \phi(w\lambda) \qquad (w\in W, \lambda\in \af_\C^*) \, .
\end{equation}

Let $\PW_W(\af_\C^*, C^\infty(K/M))$ be the subspace of $\PW(\af_\C^*, C^\infty(K/M))$ of Paley-Wiener functions that satisfy
all intertwining relations (\ref{I-relation}).
Then Helgason's Paley-Wiener theorem \cite[Theorem 8.3]{H2} states that
\begin{equation} \label{H-PW}
\F(C_R^\infty(Z))
=\PW_W(\af_\C^*, C^\infty(K/M))_R \qquad (R>0)\, .
\end{equation}

\subsection{Intertwining relations on $K$-types}
For any $\tau\in\hat{K}$ and $\lambda\in\af_{\C}^{*}$ we have
\begin{equation}\label{eq identification V_lambda[tau]}
V_\lambda[\tau]=C^\infty(K/M)[\tau]= V_\tau \otimes V_{\tau^\vee}^M
\end{equation}
as $K$-modules, where $V_{\tau^{\vee}}=V_{\tau}^{*}$.
We denote by $J_{w,\lambda}[\tau]$ the restriction of $J_{w,\lambda}$ to
$V_\lambda[\tau]$ and observe that $J_{w,\lambda}[\tau]:V_{\lambda}[\tau]\to V_{w\lambda}[\tau]$. Within the identification (\ref{eq identification V_lambda[tau]}) we then obtain
$$J_{w,\lambda}[\tau]\in \End_K (V_\tau \otimes V_{\tau^\vee}^M) \simeq \End(V_{\tau^\vee}^M)\, .$$

\par Next we recall Kostant's factorization of $J_{w,\lambda}[\tau]$.   In general, if $\ef\subset \gf$ is a subspace, we
denote by $\Sc(\ef)$ the symmetric algebra and by $\Sc^{\star}(\ef)$  the image of $\Sc(\ef)$ in $\U(\gf)$ under
the symmetrization map. From the Cartan decomposition $\gf=\sf+\kf$ and the PBW-theorem we thus obtain the direct sum decomposition
$$
\U(\gf)
= \Sc^{\star}(\sf) \oplus U(\gf)\kf\, .
$$
Next, according to \cite[Th.~15]{KR} we have $\Sc(\sf)= \Hc(\sf)\otimes \Ic(\sf)$, where $\Hc(\sf)$ denotes the harmonic polynomials
on $\sf_\C^*$ and $\Ic(\sf)$ the $K$-invariant polynomials on $\sf_\C^*$. We derive the refined decomposition
\begin{equation}\label{eq U(g)-decomposition}
\U(\gf)
= \Hc^{\star}(\sf) \Ic^{\star}(\sf) \oplus \U(\gf) \kf\ .
\end{equation}
Consequently we have for all $\lambda\in \af_\C^*$ that
$$ d\pi_\lambda(\U(\gf))v_{K,\lambda} = d\pi_\lambda(\Hc^{\star}(\sf)) v_{K,\lambda}\, .$$
We recall from Lemma \ref{lemma quotient} that in case $\lambda$ satisfies the Kostant condition \eqref{Kos}, the vector $v_{K,\lambda}$ is cyclic in $V_\lambda$ for $\U(\gf)$. In general we have for each $\tau\in \hat K$ the $K$-equivariant maps
$$ Q_\tau(\lambda): \Hc^{\star}(\sf)[\tau]\to V_\lambda[\tau]= V_\tau\otimes V_{\tau^\vee}^M, \ \ D\mapsto d\pi_\lambda(D)v_{K,\lambda}\ ,$$
which are isomorphisms if $\lambda$ satisfies \eqref{Kos}, see \cite[Cor.~to Prop.~4  and Cor.~to Th.~7]{Kos}.
(In \cite{Kos} the polynomials $Q_{\tau}$ are denoted by $P^{\tau}$. Compared to the polynomials defined in \cite[p. 238]{H3} there is a sign difference in the argument.)

For fixed $\tau\in \hat K_M$ we recall that the assignment
$$\af_\C^* \ni \lambda \to Q_\tau(\lambda)\in \Hom_K(   \Hc^{\star}(\sf)[\tau], V_\tau \otimes V_{\tau^\vee}^M )$$
is polynomial.
Since $J_{w,\lambda} v_{K,\lambda} = v_{K, w\lambda}$, we obtain the relation
$$
J_{w,\lambda}[\tau] \circ Q_\tau( \lambda) = Q_\tau({w\lambda})\ ,
$$
and as a consequence Kostant's factorization
\begin{equation} \label{factor tau} J_{w,\lambda}[\tau]= Q_\tau( w\lambda)\circ Q_\tau(\lambda)^{-1}\ ,\end{equation}
which exhibits $J_{w,\lambda}[\tau]$ for fixed $\tau\in \hat K_M$ as a rational vector-valued  function
$$\af_\C^* \ni \lambda \mapsto J_{w,\lambda}[\tau]\in \End(V_{\tau^\vee}^M)\, .$$

\begin{rmk}\label{Rem Q polynomials}
To understand the polynomial dependence of $\lambda\mapsto Q_\tau(\lambda)$ better, it proves
useful to introduce a normalization.
Set
$$ \tilde Q_\tau(\lambda):=Q_\tau(\lambda)\circ Q_\tau(0)^{-1} \in \End_K (V_\tau\otimes
V_{\tau^\vee}^M)\simeq \End(V_{\tau^\vee}^M)\, .$$
Hence $\tilde Q_\tau(0)=\id$ and we can, upon fixing a basis
of the vector space $V_{\tau^\vee}^M$, view $\tilde{Q}_{\tau}$ as a polynomial function on $\af_{\C}^{*}$ with values in the space of  $l(\tau)\times l(\tau)$-matrices, where $l(\tau):= \dim V_{\tau^\vee}^M$.
\end{rmk}

\begin{rmk} In case $G$ has real rank one, the subgroup $M\subset K$ is symmetric and thus
$V_{\tau^{\vee}}^M$ is one-dimensional for all $\tau\in \hat K_M$. In this case, for fixed $\tau\in \hat K_M$ the map
$$\lambda\mapsto \tilde Q_\tau(\lambda)$$
is an explicitly computable polynomial in $\lambda$ (see \cite[Ch.~III, Cor.~11.3]{H3}), and consequently $\lambda\mapsto J_{w, \lambda}[\tau]$ is a scalar-valued rational function.
\par Specifically, let now $G=\Sl(2,\R)$ with $K=\SO(2,\R)$ and $A$ as before.  We identify $\hat K_M$ with $\Z$
and $\af_\C^*$  with $\C$ via $\C \ni\lambda\mapsto \lambda \rho$.
Then for $n=\tau\in \Z$
\begin{align*}
\tilde Q_n(\lambda)
&= \frac{\Gamma\left( \frac{1}{2} (i\lambda + \rho)(\alpha^\vee) +|n| \right)\Gamma\left(\frac{1}{2}\rho(\alpha^\vee)\right) }{
\Gamma\left( \frac{1}{2} (i\lambda + \rho)(\alpha^\vee)\right)\Gamma\left(\frac{1}{2} \rho(\alpha^\vee)+|n|\right) }
=\frac {\Gamma\left(\frac{1}{2}  (i\lambda + 1) +|n| \right)\Gamma\left(\frac{1}{2}\right)}
    {\Gamma\left( \frac{1}{2}  (i\lambda + 1)\right)\Gamma\left(\frac{1}{2}  +|n|\right)}\\
&= \frac{(1+i\lambda) ( 3+i\lambda)\cdot\ldots \cdot (2|n|-1 +i\lambda)}{1\cdot 3\cdot\ldots \cdot (2|n|-1)}\, .
\end{align*}
Then, for all $n\in \Z=\hat K_M$ and $\lambda\in \C=\af_\C^*$ and $w\in W$ the non trivial element, the map $J_{w,\lambda}[n]$ is given by the scalar
$$
J_{w,\lambda}[n]
=\frac{(1-i\lambda) ( 3-i\lambda)\cdot\ldots \cdot (2|n|-1 -i\lambda)}{(1+i\lambda) ( 3+i\lambda)\cdot\ldots \cdot (2|n|-1 +i\lambda)}\ .
$$
\end{rmk}

\section{Strategy of proof}\label{strategy}
In this section we describe the general strategy of proof for the analytic inclusion.
The approach is simpler when $G/K$ has rank one, and therefore we give a separate proof for that.
The strategy for rank one is described through the following Ansatz~1. The general case is treated in Ansatz~2.

\subsection{Ansatz 1}\label{subsection Ansatz 1}
We consider a module $V_{\lambda_0}$, where $\lambda_0$ satisfies \eqref{Kos}.
Let $r>0$ and $v\in V_{\lambda_0}^\omega(r)$, i.e.~$v=\sum_{\tau \in \hat K_M} v_\tau$ with $v_\tau\in V_{\lambda_0}[\tau]= V_\tau \otimes V_{\tau^\vee}^M$,
so that
$$
\sum_{\tau\in \hat K_M} e^{r'|\tau|} \|v_\tau\|_\tau<\infty
\qquad(0<r'<r)\,.
$$
We make the following ansatz. First, let
$$
F(\lambda)
= F_v(\lambda) = \sum_{\tau\in \hat K_M} u_\tau(\lambda)\ ,
$$
where
$$
\af_\C^* \ni \lambda\to u_\tau(\lambda)\in V_\tau\otimes V_{\tau^\vee}^M
$$
is a certain holomorphic function  such that $u_\tau(\lambda_0)=v_\tau$.  Specifically, we set
$$
u_\tau(\lambda)
= \phi_\tau(\lambda) Q_\tau(\lambda) \circ Q_\tau(\lambda_0)^{-1} v_\tau\ ,
$$
where $\phi_\tau\in \Oc(\af_\C^*)^W$ is a $W$-invariant holomorphic function with $\phi_\tau(\lambda_0)=1$.  Suppose that
the series defining $F(\lambda)$ converges locally uniformly, so that $F_v\in \Oc(\af_\C^*, C^\infty(K/M))$.
Then we observe with \eqref{factor tau} and the $W$-invariance of $\lambda \mapsto \phi_\tau(\lambda)$ that
\begin{align*}
 J_{w,\lambda} F(\lambda)
&= \sum_{\tau\in \hat K_M}\phi_\tau(\lambda) J_{w,\lambda}[\tau] \circ Q_\tau(\lambda) \circ Q_\tau(\lambda_0)^{-1}  v_\tau\\
&= \sum_{\tau\in \hat K_M} \phi_\tau(\lambda)
Q_\tau (w\lambda)\circ Q_\tau(\lambda_0)^{-1} v_\tau
=  F(w\lambda)\ .
\end{align*}
In other words $\lambda\mapsto F(\lambda)$ satisfies the intertwining relations. If we can now construct the
$\phi_\tau$ in such a way that
$F\in \PW(\af_\C^*, C^\infty(K/M))_R$
for some $R=R(r)$, then the Paley-Wiener theorem \eqref{H-PW} implies the existence of an
$f \in C_R^{\infty}(Z)$
such that $\F(f) =F$. In particular, we obtain $v = \pi_{\lambda_0}(f) v_{K, \lambda_0}$, that is

$$
V_{\lambda_0}^\omega(r)
\subset C_R^{\infty}(G) * V_{\lambda_0}\, .
$$

We follow this ansatz for the rank $1$ spaces in Section \ref{Section rank one cases}.

\subsection{Ansatz 2}\label{subsection Ansatz 2}  For the second ansatz we need some terminology.
We denote by $\Mf(\af_\C^*, C^\infty(K/M))$ the space of $C^\infty(K/M)$-valued meromorphic
functions on $\af_\C^*$. We recall that a vector-valued function $f$ on $\af_{\C}^{*}$ is called meromorphic provided that for all $\lambda_{0}\in\af_{\C}^{*}$ there exists an open neighborhood $U$ of $\lambda_{0}$ and a polynomial $p(\lambda)$ so that $\lambda\mapsto p(\lambda)f(\lambda)$ extends to a holomorphic function on $U$.
In this regard we recall that $\Hc_\lambda^\infty=C^\infty(K/M)$
as $K$-modules for every $\lambda\in \af_\C^*$.  We then view an element
$f\in \Mf(\af_\C^*, C^\infty(K/M))$ as a section of the bundle $\coprod_{\lambda\in \af_\C^*} \Hc_\lambda^\infty\to \af_\C^*$, i.e.~we consider $f(\lambda)\in \Hc_\lambda^\infty$.
The key observation is that the prescription
$$
W \times \Mf(\af_\C^*, C^\infty(K/M))\to \Mf(\af_\C^*, C^\infty(K/M)), \ \ (w, f)\mapsto w\circ f;
$$
$$
(w \circ f )(\lambda):= J_{w, w^{-1}\lambda} f (w^{-1} \lambda)\qquad (\lambda\in \af_\C^*)\, .
$$
defines an action of $W$ and, moreover,  a meromorphic function $f$
satisfies the intertwining relations if and only if
it is $W$-invariant for this action.

Now we come to the ansatz proper. Fix $\lambda_0\in \af_\C^*$ which satisfies
the Kostant condition \eqref{Kos}, and let $W_{\lambda_0}\subset W$ be the stabilizer of $\lambda_0$.
As $\lambda_0$ satisfies \eqref{Kos}, it follows that $J_{w,w^{-1}\lambda_0}= J_{w, \lambda_0}$ is defined for all $w\in W_{\lambda_0}$ and constitutes
an intertwining operator $J_{w, \lambda_0}: V_{\lambda_0}^\infty\to V_{\lambda_0}^\infty$ with
$J_{w,\lambda_0}(v_{K, \lambda_0})= v_{K,\lambda_0}$. The fact that  $v_{K,\lambda_0}$ is fixed by $J_{w,\lambda_0}$ and that
$v_{K,\lambda_0}$ is cyclic for $V_{\lambda_0}$ (see Lemma \ref{lemma quotient}) implies
that $J_{w,\lambda_0}$ is equal to the identity on $V_{\lambda_{0}}$ and hence also on $V^{\infty}_{\lambda_{0}}$.

Let now $v \in V_{\lambda_0}^\omega(r)$. Let further $f_v:\af_\C^*\to C^\infty(K/M)$ be a
holomorphic function satisfying the properties
\begin{itemize}
\item $ f_{v}(\lambda_0)= \frac{1}{|W_{\lambda_0}|} v$,
\item $f_{v}(w\lambda_0)=0$ if $w\in W\setminus W_{\lambda_{0}}$.
\end{itemize}
Given a choice for $f_v$, we define a meromorphic function  by
$
 \Ac(f_v)
 :=\sum_{w\in W} w \circ f_v
$,
and note that $\Ac(f_v)$ automatically satisfies the intertwining relations.
Moreover,
\begin{equation}\label{property of Ansatz 2}
\Ac(f_v)(\lambda_0)
=\sum_{w\in W} J_{w, w^{-1}\lambda_0}f_v(w^{-1}\lambda_0)
=\sum_{w\in W_{\lambda_0}}  J_{w, \lambda_0} f_v(\lambda_0)
= v\, ,
\end{equation}
i.e.~$\Ac(f_v)$ interpolates $v$ at $\lambda=\lambda_0$.

The difficulty  is that the operators $J_{w, w^{-1}\lambda}$ have poles and the function $f_{v}$ has to be chosen carefully, so that $\Ac(f_v)$ is indeed holomorphic and satisfies
the Paley-Wiener condition for some $R=R(r)>0$. The overall strategy
is to start with a simple minded function $f_{v}(\lambda) = p_{\lambda_0}(\lambda)v $ for some polynomial
$p_{\lambda_0}$ and then modify $f_{v}$ along its $K$-isotypical components, i.e.~ for each $\tau\in\hat{K}$ we replace $p_{\lambda_0} (\lambda)v_\tau$ by
$\phi_\tau(\lambda) p_{\lambda_0}(\lambda)v_\tau$ for some appropriate
holomorphic function $\phi_\tau$.

This ansatz is used for the general case in Section \ref{Section higher rank case}.

\begin{rmk} Compared to the first ansatz this approach is computationally more complex, as we have to average over the Weyl group $W$, and in addition the functions $\phi_{\tau}$ have to be such that the poles of the rational functions $J_{w,\lambda}$ are canceled. However, the advantage of this
ansatz is that intertwining operators, in contrast to the $Q$-polynomials, factor into rank one intertwiners, which can be explicitly
computed and estimated.
\end{rmk}

\subsection{An application of Helgason's Paley-Wiener theorem}

The following proposition will be used in
the implementation of both ansatzes.

We define $\Hc$ to be the $K$ representation $L^{2}(K/M)$. Accordingly, we write $\Hc^{\infty}$ and $\Hc^{\omega}$ for $C^{\infty}(K/M)$ and $C^{\omega}(K/M)$, respectively. For each $r>0$ we define a Fr\'echet space by
$$
\Hc^{\omega}(r)
:=\{v\in \Hc^{\omega}\mid (\forall 0<r'<r)\  \sum_{\tau\in\hat{K}_{M}}e^{r'|\tau|}\|v_{\tau}\|<\infty\}
$$
with the indicated seminorms. We write $\Hc_{\tau}$ for the $\tau$-component of $\Hc$.

\begin{prop}\label{Prop inverse PW}
Let $r,R>0$ and $\lambda_{0}\in \af_{\C}^{*}$. Consider a family $(F_{\tau})_{\tau\in\hat{K}_{M}}$
of holomorphic functions $F_\tau:\af_{\C}^{*}\to \End(\Hc_{\tau})$ satisfying the following conditions.
\begin{enumerate}[(i)]
\item\label{Prop inverse PW - assumption 1} For every $\tau\in\hat{K}_{M}$ we have $F_{\tau}(\lambda_{0})=\id$.
\item\label{Prop inverse PW - assumption 2} For every $\tau\in\hat{K}_{M}$, $w\in W$ and $\lambda\in\af_{\C}^{*}$ the intertwining relation holds
$$
F_{\tau}(w\cdot \lambda)
=J_{w,\lambda}\circ F_{\tau}(\lambda).
$$
\item\label{Prop inverse PW - assumption 3}
There exist a real number $0<r'<r$, integers $j,l\in\N_{0}$, and a constant $C'>0$ so that for all $\tau\in\hat{K}_{M}$ and $\lambda\in \af_{\C}^{*}$
$$
\|F_{\tau}(\lambda)\|_{\mathrm{op}} \leq C'  (1+|\tau|)^{j}\,e^{r' |\tau|}\,(1+\|\lambda\|)^l\,e^{ R\, \|\im \lambda\|}\,.
$$
\end{enumerate}
Then for every $\epsilon>0$ there exists a continuous linear map $\varphi: \Hc^{\omega}(r)\to C_{R+\epsilon}^{\infty}(G/K)$ so that
$$
v=\varphi(v)*v_{K,\lambda_{0}}
\qquad\big(v\in\Hc^{\omega}(r)\big) \,.
$$
\end{prop}

\begin{proof}
For $k\in\N_{0}$, let $p_{k}$ be the continuous seminorm on $\Hc^{\infty}$ given by
$$
p_{k}(u)
=\sum_{\tau\in\hat{K}_{M}}(1+|\tau|)^{k}\|u_{\tau}\|
\qquad\big(u\in \Hc^{\infty}\big)\,.
$$
Note that this family of seminorms determines the topology of $\Hc^{\infty}$.

Let  $r'<r''<r$.
It follows from (\ref{Prop inverse PW - assumption 3}) that for $\lambda\in\af_{\C}^{*}$, $k\in\N_{0}$ and $v\in\Hc^{\omega}(r)$
\begin{align}\label{eq estimate sum F_n}
\sum_{\tau\in\hat{K}_{M}}(1+|\tau|)^{k}\|F_{\tau}(\lambda)(v)\|
\leq C''(1+\|\lambda\|)^l\,e^{R\,\|\im \lambda\|}\sum_{\tau\in\hat{K}_{M}}e^{r''|\tau|}\|v_{\tau}\|\,,
\end{align}
where
$$
C''
:=C'\sup_{\tau\in\hat{K}_{M}}(1+|\tau|)^{j+k}e^{(r'-r'')|\tau|}<\infty.
$$
By definition $\sum_{\tau\in\hat{K}_{M}}e^{r''|\tau|}\|v_{\tau}\|<\infty$ for
each $v\in \Hc^{\omega}(r)$. It follows from  \eqref{eq estimate sum F_n}
that for every $v\in \Hc^{\omega}(r)$ the series
\begin{equation*}
F_{v}(\lambda)
:= \sum_{\tau\in\widehat{K}_{M}}  F_{\tau}(\lambda)(v_{\tau})\qquad\big(\lambda\in \af_{\C}^{*}\big)
\end{equation*}
converges in $\Hc^{\infty}$. The convergence is uniform for $\lambda$ in compacta, and hence $F_{v}$ is a holomorphic $\Hc^\infty$-valued function depending linearly on $v$.

Let $\epsilon>0$. We claim that there exists a $\theta\in C_{\epsilon}^{\infty}(K\bs G/K)$ so that $\theta* v_{K,\lambda_{0}}=v_{K,\lambda_{0}}$. To see this, first note that
$$
C_{c}^{\infty}(K\bs G/K)*v_{K,\lambda_{0}}
\subseteq\C v_{K,\lambda_{0}}.
$$
Consider now a Dirac sequence $(\theta_{n})_{n\in\N}$ of functions $\theta_{n}\in C_{1/n}^{\infty}(K\bs G/K)$. Since $\theta_{n}*v_{K,\lambda_{0}}$ converges to $v_{K,\lambda_{0}}$ for $n\to\infty$, there exists an $m\in\N$ so that $\theta_{n}*v_{K,\lambda_{0}}\neq 0$ for all $n>m$.  Let now $n>m$ be so large that $\frac{1}{n}<\epsilon$. After a rescaling of $\theta_{n}$ we obtain a function with the claimed property.

Since $\theta$ is $K$-invariant, its Fourier transform  $\hat{\theta}=\F(\theta)$ is a $W$-invariant scalar-valued
holomorphic function
on $\af_\C^*$, and by the Paley-Wiener Theorem (\ref{H-PW}) it satisfies for every $N\in\N_{0}$ the estimate
\begin{equation}\label{eq PW estimate theta}
\sup_{\lambda\in \af_\C^*}  |\hat{\theta}(\lambda)|  (1 +\|\lambda\|)^N e^{-\epsilon \|\im \lambda\|} <\infty\,
\end{equation}
Moreover, since $\theta*v_{K,\lambda_{0}}
=v_{K,\lambda_{0}}$ we have
\begin{equation}\label{eq Ftheta(lambda0)=1}
\hat{\theta}(\lambda_{0})=1\,.
\end{equation}

For $\lambda\in \af_{\C}^{*}$ and $v\in\Hc^{\omega}(r)$ we define $f_{v}(\lambda):=\hat{\theta}(\lambda) F_{v}(\lambda)\in \Hc^{\infty}$. The function $f_{v}:\af_{\C}^{*}\to\Hc^{\infty}$ thus obtained is holomorphic.
It follows from (\ref{eq Ftheta(lambda0)=1}) and
assumption (\ref{Prop inverse PW - assumption 1}) that $f_{v}(\lambda_{0})=F_{v}(\lambda_{0})=v$. In view of assumption  (\ref{Prop inverse PW - assumption 2}) the function $f_{v}$ satisfies  the intertwining relations (\ref{I-relation}). Finally, it follows from the estimates (\ref{eq estimate sum F_n}) and (\ref{eq PW estimate theta}) that there exist for every $N\in\N_{0}$ and $k\in\N_{0}$ a constant $C_{N,k}>0$, so that for every $\lambda\in\af_{\C}^{*}$ and $v\in\Hc^{\omega}(r)$
\begin{equation}\label{eq PW estimate f}
p_{k} \big(f_{v}(\lambda)\big)\leq C_{N,k} (1 +\|\lambda\|)^{-N} e^{(R+\epsilon) \|\im \lambda\|} \sum_{\tau\in\hat{K}_{M}}e^{r''|\tau|}\|v_{\tau}\|\,.
\end{equation}
Now it follows from the Paley-Wiener theorem (\ref{H-PW}) that $f_{v}=\F(\varphi'_{v})$ for some $\varphi'_{v}\in C_{R+\epsilon}^{\infty}(G/K)$. Set $\varphi_{v}=\varphi'_{v}* \theta\in C^{\infty}_{R+2\epsilon}(G/K)$. Note that $\varphi_v$ depends linearly on $v$ and satisfies
$$
\varphi_{v}*v_{K,\lambda_{0}}
=\varphi_{v}'*v_{K,\lambda_{0}}
=\F(\varphi'_{v})(\lambda_{0})
=f_{v}(\lambda_{0})
=v\,.
$$

It remains to show continuity from $\Hc^{\omega}(r)$ to $C_{R+2\epsilon}^{\infty}(G/K)$ of the map $v\mapsto\varphi_v$.
The Paley-Wiener space $\PW_W(\af_\C^*, \Hc^\infty)_{R+\epsilon}$ is a subspace of $L^{2}\big(\af^{*},\Hc, \frac{d\lambda}{|c(\lambda)|^{2}}\big)$. By the Plancherel theorem for $G/K$ and (\ref{eq PW estimate f}) we have
$$
\|\varphi_{v}'\|_{L^{2}}^{2}
=\int_{\af^{*}}\|f_{v}(\lambda)\|^{2}\frac{d\lambda}{|c(\lambda)|^{2}}
\leq \int_{\af^{*}}p_{0} \big(f_{v}(\lambda)\big)^{2}\frac{d\lambda}{|c(\lambda)|^{2}}
\leq c_{0}\left(\sum_{\tau\in\hat{K}_{M}}e^{r''|\tau|}\|v_{\tau}\|\right)^{2}\ ,
$$
with
$$
c_{0}
=C_{N,0}^{2}\int_{\af^{*}}(1 +\|\lambda\|)^{-2N} \frac{d\lambda}{|c(\lambda)|^{2}}<\infty
$$
for a sufficiently large $N\in\N$. Finally
for every continuous seminorm $q$ on $C_{R+2\epsilon}^{\infty}(G/K)$ there exists a constant $c'>0$, only depending on $\theta$, so that
$$
q(\varphi_{v})
\leq c'\|\varphi'_{v}\|_{L^{2}}.
$$
The continuity follows.
\end{proof}

\section{An explicit construction in one variable}\label{Section construction}

For every $n\in\N_{0}$ and $R>0$  we define an entire function $f_{n,R}$ on $\C$ by
\begin{equation} \label{def f_n}
f_{n,R}(z)
:= \frac{\sin (z R\pi)}{z R\pi\cdot\prod_{j=1}^{n}\left(1-\left(\frac{Rz}{ j}\right)^2\right) }
= \prod_{j=n+1}^{\infty}\left(1-\left(\frac{Rz}{ j}\right)^2\right)
\qquad(z\in\C)\, ,
\end{equation}
invoking the product expansion of the sine function.
Next we define for $n\in\N_{0}$ a polynomial  function $q_{n}$ on $\C$ by
\begin{equation} \label{def q_n}
q_n(z)
:= \prod_{j=1}^{n}\Big(1+\frac{z}{j}\Big)
\qquad(z\in\C)\, .
\end{equation}

\begin{proposition} \label{prop basic estimate}
There exist $ c, R_0>0$ and for every $r>0$
a constant $C_{r}>0$  so that the following assertion holds for
every $n\in\N_{0}$.

Let $r>0$ and $R>R_0$ with
\begin{equation}\label{eq relation between R and r}
\frac{(\log R)^2}{R^2} <c r.
\end{equation}
Let $V$ be a finite dimensional inner product space and $P: \C\to\End(V)$
a polynomial map such that
$$
\|P(z)\|_{\mathrm{op}}
\leq q_{n}(|z|)^k
\qquad(z\in\C)
$$
for some $k\in\N$.
Then
$$
\|f_{n,R}(z)^k P(z)\|_{\mathrm{op}} \leq  \left[ C_{r} \,e^{rn}e^{ R \pi\,|\im z|} \right]^k
$$
for all $z\in\C$.
\end{proposition}

The proof is divided into several lemmas. Let $$F_{n,R}(z):=f_{n,R}(z)^k P(z).$$
The first two lemmas contain estimates of the function defined by
$$\tilde F_{n,R}(z):=f_{n,R}(z) q_{n}(|z|) \,,$$
for which we have
\begin{equation}\label{eq tilde F dominates}
\|F_{n,R}(z)\|_{\mathrm{op}}
\leq |\tilde F_{n,R}(z)|^k
\qquad(z\in\C)\,.
\end{equation}

\begin{lemma} \label{largelambda}There exists a constant
$C>0$, so that for all $R>3$, $n\in\N_{0}$ and $z\in\C$ with $|z|\geq \frac{n}{R}$ we have
$$
 |\tilde F_{n,R}(z)|
 \leq Ce^ {R\pi |\im z|}\, .
 $$
\end{lemma}

\begin{proof} By symmetry, we may assume without loss of generality that $\re z\geq 0$.
Note that
\begin{equation}\label{eq expression for F_(n,R)}
\tilde F_{n,R}(z)
= \frac{(1+|z|)\cdots(n+|z|)}{ (1+Rz)\cdots(n+Rz)}
\cdot \frac{1 \cdot 2\cdots n }{ (1-Rz)\cdots (n-Rz)}
\cdot \frac{\sin(\pi Rz)}{\pi Rz}\,.
\end{equation}
We claim that
$$
\left|\frac{(1+|z|)\cdots(n+|z|)}{(1+Rz)\cdots (n+Rz)}\right|
\leq 1\, .
$$
To prove the claim it suffices to show that for all $1\leq j\leq n$ we have
$$
j+|z|
\leq |j+Rz|\ .
$$
Since $\re z\geq0$, we have
$$ | j + Rz|^2 = R^2 |z|^2 + 2Rj \re z + j^2\geq R^2|z|^2 + j^2\ .$$
For $R\geq 3$ and $|z|\geq \frac{n}{R}$ the condition $ R^2|z|^{2}  \geq |z|^{2}+2n|z|$ is satisfied, and hence
$$
R^2|z|^2 + j^2
\geq |z|^2 + 2n |z| + j^2
\geq (j+|z|)^2 \ .
$$
This proves the claim.

We further claim that
$$
\left|\frac{1 \cdot 2\cdots (n-1)}{(1-Rz)\cdots (n-1-Rz)}\right|
\leq 1
$$
if $R|z|\geq n$. This is a direct consequence of the inequality $|j-Rz|\geq R|z|-j \geq n-j$.

Altogether, we obtain the estimate
\begin{align*}
|\tilde F_{n,R}(z)|
&\leq \left|\frac{n}{n-Rz}\right| \cdot \left|\frac{\sin(\pi Rz)}{\pi Rz}\right|\\
&=\frac{n}{R|z|} \left|\frac{\sin\big(\pi Rz\big)} {\pi(n-Rz)}\right|
\leq\left|\frac{\sin\big(\pi (n-Rz)\big)} {\pi(n-Rz)}\right|
\leq C e^ {R\pi |\im z|} \ .
\end{align*}
\end{proof}

\begin{lemma}\label{Lemma estimate small lambda}
There exists a constant $c>0$ such that the following holds: For all $r >0$ there exists  $C>0$, so that for all $n\in\N_{0}$ and $R>e$ with
$$
\frac{(\log R)^{2}}{R^{2}}<c r\ ,
$$
we have
$$
|\tilde F_{n,R}(z)|
\leq C e^{r n}
\qquad (0 \leq z \leq{\textstyle \frac{n}{R}})\ .
$$
\end{lemma}

\begin{proof}
Let $n\in\N_0$, $R>1$ and $z\geq 0$.
We shall estimate $\tilde F_{n,R}(z) $ using Stirling's approximation.  Euler's reflection identity $\Gamma(1 -x) \Gamma(x) =
\frac{\pi}{\sin(\pi x)}$ and the functional equation of the Gamma function yield
$$
\left(\prod_{j=1}^{n}(j+Rz)\right)\left(\prod_{j=1}^{n}(j-Rz)\right)\frac{\pi Rz}{\sin \pi Rz}
=\Gamma(n+1-Rz)\Gamma(n+1+Rz).
$$
This allows to rewrite (\ref{eq expression for F_(n,R)}) and express $\tilde F_{n,R}$ in terms of Gamma functions:
$$
\tilde F_{n,R}(z)
=\frac{\Gamma(n+1+z)\Gamma(n+1)}{\Gamma(1+z)\Gamma(n+1+Rz)\Gamma(n+1-Rz)}\, .
$$
We recall Stirling's approximation
\begin{equation}\label{eq Stirling}
\Gamma(x)
=\sqrt{\frac{2\pi}{x}} e^{-x} x^x\left(1+ O(1/x)\right)
\end{equation}
for $x \to \infty$.
Applying Stirling to $\tilde F_{n,R}$, we obtain that there exists a constant $c'>0$, independent of $n$ and $R$, so that for all
$z\in[0,\frac{n}{R}]$
\begin{equation}\label{eq first F estimate}
\tilde F_{n,R}(z)
\leq c'\sqrt{1+\frac{n}{R}}\,e^{h_{n,R}(z)},
\end{equation}
where
\begin{align*}
h_{n,R}(z)
&:=(n+1+z)\log(n+1+z)+(n+1)\log(n+1)\\
    &\qquad-(1+z)\log(1+z)-(n+1+Rz)\log(n+1+Rz)\\
    &\qquad-(n+1-Rz)\log(n+1-Rz)\, .
\end{align*}
Here we used the straightforward estimate for $z\in[0,\frac{n}{R}]$
$$
\frac{(1+z)(n+1+Rz)(n+1-Rz)}{(n+1+z)(n+1)}
=\frac{(1+z)\big((n+1)^{2}-R^{2}z^{2}\big)}{(n+1+z)(n+1)}
\leq 1+\frac{n}{R}
$$
to estimate the square roots in (\ref{eq Stirling}).

The term $(1+z)\log(1+z)$ in $h_{n,R}(z)$ can be estimated below by $z\log(z)$.
With the substitution $z=(n+1)x$ we then obtain
\begin{equation}\label{eq for h and H}
h_{n,R}(z)\le (n+1) H_{R}(x)
\end{equation}
where
\begin{align*}
H_R(x)
&:=(1+x)\log(1+x)-x\log(x)\\
    &\qquad-(1+Rx)\log(1+Rx)-(1-Rx)\log(1-Rx)\, .
\end{align*}
We will show
\begin{equation}\label{ineq for H_R}
H_R(x)\leq 4 \frac{(\log R)^2}{R^2}, \quad x\in (0,\tfrac 1R),
\end{equation}
for all $R\ge e$.
We estimate the first two terms of $H_R(x)$ in a separate lemma.

\begin{lemma} \label{lemma ineq with b}
For every $b\ge e^2$ and $0< x\le 1$
\begin{equation*}
(1+x)\log(1+x)-x\log x \leq \frac{(\log b)^2}{b}+b x^2.
\end{equation*}
\end{lemma}

\begin{proof}
Since $(1+x)\log(1+x)\le 2x$ for $0<x\le 1$ it suffices to show
\begin{equation*}
2x-x\log x=:\phi(x)\leq \psi(x):=\frac{(\log b)^2}{b}+b x^2\, .
\end{equation*}
The functions $\phi$ and $\psi$ are concave and convex, respectively. We will prove the inequality by exhibiting
a separating line of slope $\log b$.

We have $\phi'(x)=1-\log x$ and hence $\phi'(x)=\log b$ for $x=\frac eb$. Then
$$\phi(x)\le \phi(\tfrac eb)+\log(b)(x-\tfrac eb)=\tfrac eb+\log(b)x\, .$$
On the other hand $\psi'(x)=2bx=\log b$ for $x=\frac{\log b}{2b}$ and therefore
$$\psi(x)\ge \psi\Big(\frac{\log b}{2b}\Big)+\log (b)\Big(x-\frac{\log b}{2b}\Big)=\frac{3(\log b)^2}{4b}+\log(b)x\, .$$
Hence $\psi\ge \phi$ if $\frac34 (\log b)^2\ge e$ and in particular if $b\ge e^2$.
\end{proof}

We proceed with the proof of \eqref{ineq for H_R}. Let $$\varphi(t)=(1+t)\log(1+t)+(1-t)\log(1-t)$$
for $0\le t<1$. Then $\varphi(0)=\varphi'(0)=0$ and
$\varphi''(t)=\frac1{1+t}+\frac1{1-t}\ge 2$. Hence $$ \varphi(t) \ge t^2.$$
Then for $x\in (0,\frac1R)$
$$H_R(x)\le (1+x)\log(1+x)-x\log x  -R^2x^2.$$
We obtain \eqref{ineq for H_R} from Lemma \ref{lemma ineq with b} with $b=R^2$.

We can now finish the proof of Lemma \ref{Lemma estimate small lambda}. Let $0<c<\frac14$. If $r>0$, $R\ge e$ and $\frac{\log(R)^{2}}{R^{2}}\le cr$, then
$$
|\tilde F_{n,R}(z)|
\leq c' e^{r}\sqrt{1+n}\,e^{4crn} \leq C e^{rn}, \quad n\in\N_0, z\in [0,\tfrac nR],
$$
by (\ref{eq first F estimate}) and \eqref{eq for h and H}, with a constant $C>0$ depending only on $c$ and $r$.

\end{proof}

\begin{proof}[Proof of Proposition \ref{prop basic estimate}]
By Lemma \ref{largelambda} and
(\ref{eq tilde F dominates}) we have for all $R>3$, $n\in\N_{0}$ and $z\in\C$ with $|z|\geq \frac{n}{R}$
\begin{equation}\label{eq F estimate}
\|F_{n,R}(z)\|_{\mathrm{op}}\leq \left[ Ce^ {R\pi |\im z|}\right]^k\, .
\end{equation}
It therefore suffices to estimate $F_{n,R}(z)$  for $z$ in the disk $D = \left\{z \in \C : |z| \leq \frac{n}{R}\right\}$.

Note that
$$
\|F_{n,R}(z)\|_{\mathrm{op}}
=\sup_{\substack{v,w\in V\\\|v\|=\|w\|=1}}|\langle F_{n,R}(z)v,w\rangle|,
$$
and that the matrix coefficients $\langle F_{n,R}(z)v,w\rangle$ depend holomorphically on $z\in\C$.

Let $D_\pm=D\cap\C_\pm$ where $\C_\pm$ denotes the closed upper/lower half plane.
By the maximum modulus principle a holomorphic function in $D_\pm$ assumes its maximum modulus on
$\partial D_\pm$, i.e.~on the union of the semicircle $\partial D\cap \C_\pm$ and the segment $D \cap \R = [-\frac{n}{R}, \frac{n}{R}]$. We  apply the principle to the holomorphic function
$$\la F_{n,R}(z)v,w\ra\,e^{\pm iRk\pi z}$$ on
$D_\pm$, which by (\ref{eq F estimate})
is bounded in absolute value by $C^k$ on $\partial D\cap \C_\pm$.

On the other hand,  with $c$ as in Lemma \ref{Lemma estimate small lambda}
it follows that for all $r$ satisfying \eqref{eq relation between R and r}
there exists a constant $C_r$ such that
$|\la F_{n,R}(z)v,w\ra\,e^{\pm iRk\pi z}|$ is bounded by $[C_re^{rn}]^k$ for
$z \in [-\frac{n}{R}, \frac{n}{R}]$. Assuming as we may that $C_r\ge C$, we obtain
$$|\la F_{n,R}(z)v,w\ra\,e^{\pm iRk\pi z}|\leq [C_r e^{rn}]^k$$
for all $z\in D_\pm$. This implies the proposition.
\end{proof}

The following lemma will be used in the next two sections.

\begin{lemma} \label{lemma normalization f_n}
Let $R>0$ and $z_0\in \C$. Assume $Rz_0\notin \Z\bs\{0\}$.
Then
\begin{equation*}
\inf_{n\in\N_0}|f_{n,R}(z_0)|>0.
\end{equation*}
\end{lemma}

\begin{proof}
With
\eqref{def f_n} we observe that
$f_{n,R}(z)=0$  if and only if $Rz\in\Z$ and $|z|>\frac nR$.
In particular the assumption on $z_0$ implies $f_{n,R}(z_0)\neq 0$ for all $n\in\N_0$.

If $n\geq N:=\lceil R|z_0|\rceil$ then
$$
|f_{n,R}(z_0)|
\geq \prod_{j=n+1}^\infty \left( 1 - \Big(\frac{R|z_0|}{j}\Big)^2\right)
\geq \prod_{j=N+1}^\infty \left( 1 - \Big(\frac{R|z_0|}{j}\Big)^2\right)=f_{N,R}(|z_0|).
$$
Hence
 $$\inf_{n\in\N_0}|f_{n,R}(z_0)|\ge \min\big\{ |f_{0,R}(z_0)|,\dots,|f_{N-1,R}(z_0)|
 ,f_{N,R}(|z_0|)\big\}>0$$
by the first observation in the proof.
\end{proof}

We end this section with a remark that will be useful when Proposition \ref{prop basic estimate} is applied in Sections \ref{Section rank one cases} and \ref{Section higher rank case}.

\begin{rmk} \label{rmk scaling invariant}
Let $P(R,r)$ be any proposition depending on two variables $R,r>0$.
Then the proposition
$$\exists c,R_0>0\,\, \forall r>0,R>R_0 : \quad \Big( \,\frac{(\log R)^2}{R^2} <c r \,\,\Rightarrow\,\,P(R,r)\Big)
$$
possesses a scaling invariance. Let  $a,A>0$ and $B\in\R$. Then
$$\exists d,S_0>0\,\, \forall s>0,S>S_0 : \quad \Big(\,\frac{(\log S)^2}{S^2} <ds \,\Rightarrow\,\, P(AS+B,as)\Big)
 $$
is an equivalent proposition. This follows from the observation that there exist constants $C_{0},C_{1},C_{2}>0$ so that for all $R>C_{0}$
$$
C_{1}\frac{\log R}{R} \leq \frac{\log( AR+B)}{AR+B}\leq C_{2}\frac{\log R}{R}\,.
$$ 
\end{rmk}

\section{The rank one cases}\label{Section rank one cases}

Using the construction from the previous section we can now complete the argument in case
$G$ is of real rank one. Let $\alpha\in \Sigma^+$ be the indivisible root.
Then
$$\gf =  \gf^{2\alpha}+\gf^{\alpha} +\af +\mf +  \gf^{-\alpha} + \gf^{-2\alpha}$$
with
$\af  = \R \alpha^\vee$.  We set $m_\alpha:=\dim \gf^{\alpha}$ and $m_{2\alpha}=\dim \gf^{2\alpha}$. Then
$$\rho= \frac{1}{2}( m_\alpha + 2m_{2\alpha})\alpha\, . $$
The goal of this section is to prove the following

\begin{theorem} \label{theorem rank one} Let $G$ be a group of real rank one and $V_{\lambda_0}$ a representation of the
$K$-spherical principal series with $\lambda_0$ satisfying \eqref{Kos}. Then there exist positive constants $c, R_0>0$ independent of $\lambda_0$, such that for all $R,r>0$ with
$R>R_0$ and $\frac{(\log R)^2}{R^2} < c r$, we have a continuous embedding
$$
 V_{\lambda_0}^\omega(r) \subset
C_R^{\infty}(G)* v_{K,\lambda_0} 
=\big(V_{\lambda_{0}}\big)^{\min}_{R}\, .
$$
\end{theorem}

\begin{rmk} \label{rmk intro proof rank one}
With the theorem above we can prove Theorem \ref{intro theorem} for
groups of real rank one. Let $c$ and $R_0$ be as above. Then
Theorem \ref{theorem rank one} and Lemma \ref{lemma K-analytic} imply
that the conclusion of Theorem \ref{intro theorem} is valid for $V=V_{\lambda_0}$. For a general irreducible
Harish-Chandra module $V$ the conclusion
then follows from
Lemmas \ref{lemma reduction 1} and \ref{lemma quotient}.
\end{rmk}

In order to prepare for the proof of Theorem \ref{theorem rank one}, we introduce some new notation and collect a few facts
about real rank one groups which will be used in the sequel.
A special feature of this case is that $K/M$ is a compact symmetric space of rank one. In particular we have $ \dim V_\tau^M =1$ for all $\tau \in \hat K_M$.
Consequently, the $\tilde Q$-matrices from Remark \ref{Rem Q polynomials} are just scalar-valued polynomials.
These polynomials can be read off from \cite[Ch.~III, Th.~11.2]{H3}. We recall from op.~cit.~the non-negative integers $0\leq r \leq s$ attached to
a fixed $\tau \in \hat K_M$. From the proof of the cited theorem it follows that $r$ and $s$ have the same parity if $m_{2\alpha}\neq0$. It further follows from the equation for $s$ and $r$ on page 346 in op.~cit. that there exists an $m>0$, independent of $\tau$, such that
\begin{equation*}
s
\leq m |\tau|\, .
\end{equation*}
We may and will take $m\in\N$.
In order to express the $\tilde{Q}_\tau(\lambda)$ in an efficient way, we introduce some new notation.

For elements $0< a\leq b$ with $b-a\in\N_0$,
we define polynomials in the complex plane by
\begin{equation} \label{def Gamma ab}
\Gamma_{a,b}(z)
:=\frac{\Gamma (z+b) \Gamma(a)}{\Gamma(z+a)\Gamma(b)} = \frac{ (z+a)(z+a+1)\cdots(z+b -1)}{a(a+1)\cdots (b -1)}\, .
\end{equation}

For later reference we note the following estimates by the polynomials
introduced in
\eqref{def q_n}:
\begin{equation}\label{eq estimate Gamma_ab}
|\Gamma_{a,b}(z)|\leq \Gamma_{a,b}( |z|) \leq \,
\begin{cases}
 q_{b-a}(|z|)\qquad &\text{if } a\ge1\\[5pt]
  \frac ba\,q_{b-a}(|z|)\qquad &\text{if } a<1\,.
\end{cases}
\end{equation}
The inequality for $a\ge 1$ follows from
$$\frac{|z|+a+j}{a+j}=1+\frac {|z|}{a+j}\leq 1+\frac {|z|}{1+j}$$
for each $j\ge 0$, and the other one is
then a consequence of
$$\Gamma_{a,b}(|z|)=\frac{(|z|+a)\,b}{(|z|+b)\,a}\,\,\Gamma_{a+1,b+1}(|z|).$$

We also note that
\begin{equation}\label{eq Gamma geq 1}
|\Gamma_{a,b}(z)|\geq 1\qquad(\re(z)\geq 0).
\end{equation}

Next we define positive half integers. In case $m_{2\alpha}=0$ we set
$$
a_\tau:=\rho(\frac{\alpha^\vee}{2})=\frac{m_\alpha}{2}
\quad\text{and} \quad
b_\tau:= \rho(\frac{\alpha^\vee}{2})+s,
$$
and note that $b_\tau-a_\tau=s\in \N_0$.

For $m_{2\alpha}>0$ we first note that $\rho(\frac{\alpha^\vee}{2}) = \frac{m_\alpha}{2} +  m_{2\alpha}=:d\in\N$ is a positive
integer greater or equal to $2$, as $m_\alpha$ is even when $m_{2\alpha}>0$. Further we define  positive half integers by
$$
a_\tau^1:=\rho(\frac{\alpha^\vee}{4})=\frac{d}{2}
\quad \text{and} \quad
b_\tau^1:= \frac{1}{2}( s+ r +d )
$$
and
$$
a_\tau^2:=\frac{1}{2}(d+ 1-m_{2\alpha})
\quad \text{and} \quad
b_\tau^2:= \frac{1}{2}( s -r +d + 1 - m_{2\alpha})\, .
$$
Then both $b_\tau^1 - a_\tau^1 = \frac{1}{2}(s+r)$ and $b_\tau^2 - a_\tau^2 = \frac{1}{2}(s-r)$ are non-negative integers.

Having defined these constants, we rephrase  \cite[Ch. III,  Cor. 11.3]{H3} as follows:

\begin{lemma}\label{Lemma q polynomials}
 Let $G$ be a group of real rank one and $\tau \in \hat K_M$. Then the
following assertions hold:
\begin{enumerate} \label{lemma q in rank one}
\item If $m_{2\alpha}=0$, then $a_\tau\ge\frac 12$  and
\begin{equation*} 
\tilde{Q}_\tau(\lambda)= \Gamma_{a_\tau, b_\tau}(i\lambda(\frac{\alpha^\vee}{2}))\,.
\end{equation*}
\item If $m_{2\alpha}>0$, then
$a_\tau^1\ge a_\tau^2\ge 1$ and
\begin{equation*} 
\tilde{Q}_\tau(\lambda)= \Gamma_{a_\tau^1, b_\tau^1} (i\lambda(\frac{\alpha^\vee}{4})) \,\Gamma_{a_\tau^2, b_\tau^2} (i\lambda(\frac{\alpha^\vee}{4}))\, .
\end{equation*}
\end{enumerate}
\end{lemma}

\subsection{Proof of Theorem \ref{theorem rank one} in case of $m_{2\alpha}=0$.}
We identify $\af_\C^*$ with $\C$ via
\begin{equation*}
\C \mapsto \af_\C^*, \ \ z\mapsto z \alpha\, ,
\end{equation*}
i.e.~$\lambda=z\alpha\in \af_\C^*$ identifies with $z\in \C$.
In these coordinates we then have
\begin{equation*} 
\tilde{Q}_\tau(z)
=  \Gamma_{\frac{m_\alpha}{2},\frac{m_\alpha}{2}+s}(iz),
\end{equation*}
and it follows from \eqref{eq estimate Gamma_ab} that
$|\tilde{Q}_{\tau}(z)|
 \leq (1+2s)q_s(|z|).$

We recall that $s\leq m|\tau|$ for every $\tau\in\hat{K}_{M}$, and that $m\in\N$. We write $\lceil\tau\rceil\in\N$ for the smallest integer greater or equal than $|\tau|$. We thus obtain the bound
\begin{equation}\label{eq bound q's}
 |\tilde{Q}_{\tau}(z)|
 \leq  (1+2m\lceil\tau\rceil) q_{m\lceil\tau\rceil} (|z|).
\end{equation}

We recall the functions $f_{n,R}$, depending on $R>0$,  as defined in \eqref{def f_n}.
Let $z_0\in\C$ be so that $\lambda_0=z_{0}\alpha$.
We assume that $Rz_{0}$ is not a non-zero integer. Then $f_{n,R}(z_{0})\neq 0$ for all $n\in \N_{0}$. For $\tau\in\hat{K}_{M}$ we define the $W$-symmetric entire function
$$
\phi_\tau:\C\to\C;\quad z\mapsto
\frac{f_{ m\lceil\tau\rceil ,R}(z)}{f_{m\lceil\tau\rceil ,R}(z_0)}.
$$
Now given $\lambda_{0}=z_0\alpha\in\af_{\C}^{*}$ satisfying \eqref{Kos},  we follow Ansatz 1 in Section \ref{subsection Ansatz 1} and define
$$
\tilde{F}_\tau(z)
= \phi_\tau(z) \tilde{Q}_\tau(z) \tilde{Q}_\tau(z_0)^{-1} \in \End_K (V_\tau\otimes V_{\tau^\vee}^M)\simeq \C
\quad(z\in\C)\,.
$$
Let $F_{\tau}:\af_{\C}^{*}\to  \End_K (V_\tau\otimes V_{\tau^\vee}^M)$ be given by
$$
F_{\tau}(z\alpha)
=\tilde F_{\tau}(z)
\qquad(z\in \C).
$$
It is immediate that $F_{\tau}$ satisfies the conditions (\ref{Prop inverse PW - assumption 1}) and (\ref{Prop inverse PW - assumption 2}) in Proposition \ref{Prop inverse PW} with $\Hc_{\tau}=V_\tau\otimes V_{\tau^\vee}^M$.

We continue by investigating condition (\ref{Prop inverse PW - assumption 3}).
For that we need to control the normalizing factors $\tilde{Q}_{\tau}(z_{0})$ and $f_{m\lceil\tau\rceil,R}(z_{0})$.
By (\ref{Kos}) the real part of $iz_{0}$ is non-negative, and hence it follows from
(\ref{eq Gamma geq 1}) that
$$
|\tilde{Q}_\tau(z_0)|
\geq 1
\qquad (\tau \in \hat K_M).
$$
Likewise,  Lemma
\ref{lemma normalization f_n} gives a positive lower  bound for $|f_{m\lceil\tau\rceil,R}(z_{0})|$,
 uniformly in $\tau$.

Let $c,R_{0}>0$ be as in Proposition \ref{prop basic estimate}, and assume that $R>R_{0}$ and $\frac{(\log R)^{2}}{R^{2}}<c r$. By perturbing $R$ to a slightly smaller value we can ensure $Rz_0$ is not an integer, as assumed before.
Let  $r'<r$ be such that $\frac{(\log R)^{2}}{R^{2}}<cr'$.
From Proposition \ref{prop basic estimate} and (\ref{eq bound q's}) it follows that
there exists a constant $C>0$ so that
$$
|F_\tau(\lambda)|
\leq C(1+|\tau|)\, e^{r' m|\tau|}\, e^{\pi R\|\frac12\alpha^\vee\|\,\|\im \lambda\|}
\qquad\big(\lambda\in\af^*_\C,\tau\in\hat{K}_{M}\big)\,.
$$
By Proposition \ref{Prop inverse PW} this implies $V^\omega_{\lambda_0}(mr)\subset C^\infty_{AR+\epsilon}(G/K)*v_{K,\lambda_0}$, where $A=\pi\|\frac12\alpha^\vee\|$.
By Remark \ref{rmk scaling invariant} the continuous embedding in the theorem follows. Finally, as $v_{K,\lambda_{0}}$ is $\U(\gf)$-cyclic by Lemma \ref{lemma quotient}, we have $C^\infty_{R}(G)* v_{K,\lambda_0}=(V_{\lambda_{0}})^{\min}_{R}$.

\subsection{Proof of Theorem \ref{theorem rank one} in case of $m_{2\alpha}>0$.}\label{Subsection 2 roots}
We now identify $\af_\C^*$ with $\C$ via
$$
\C \mapsto \af_\C^*, \ \ z\mapsto 2z \alpha\, .
$$
From Lemma \ref{Lemma q polynomials} we then have
$$
\tilde{Q}_\tau(z)
= \Gamma_{a_\tau^1, b_\tau^1} (iz) \Gamma_{a_\tau^2, b_\tau^2} (iz).
$$
As before we  apply \eqref{eq estimate Gamma_ab}.  The result is now
$$
\tilde{Q}_\tau (- i |z|)
\leq  [q_{m \lceil\tau\rceil} (|z|)]^2\ .
$$
Next we define the $W$-symmetric entire function
$$ \phi_\tau(z) := \frac{[f_{ m\lceil\tau\rceil }(z)]^2}{[f_{ m\lceil\tau\rceil }(z_0)]^2}$$
and argue along the same lines as before.
This concludes the proof of Theorem \ref{theorem rank one}.

\section{The general higher rank case}\label{Section higher rank case}

The goal of this section is to prove the following

\begin{theorem} \label{main theorem}Let $V_{\lambda_0}$ be a representation of the
$K$-spherical principal series with $\lambda_0$ satisfying \eqref{Kos}. Then there exist positive constants $c, R_0$ independent of $\lambda_0$, such that for all $R,r>0$ with
$R>R_0$ and $\frac{(\log R)^2}{R^2} < c r$ we have a continuous embedding
$$
V_{\lambda_0}^\omega(r)
\subset C^\infty_R(G)* v_{K,\lambda_0}
=\big(V_{\lambda_{0}}\big)^{\min}_{R}\,.
$$
\end{theorem}

\begin{rmk} \label{rmk intro proof} By the arguments in Remark \ref{rmk intro proof rank one} we obtain Theorem \ref{intro theorem} of the introduction from Theorem \ref{main theorem} together with the reduction in Section~\ref{Section Reduction}.
\end{rmk}

To prepare for the proof of Theorem \ref{main theorem} we determine some estimates of the
intertwining operators $J_{w,\lambda}$. We start by recalling the standard procedure by which
the study of $J_{w, \lambda}$ is reduced to rank one.

\subsection{Factorization of intertwining operators}
Let $w\in W$ and write
$$w= s_1  s_2 \cdots  s_n$$
as a reduced expression with simple reflections $s_i$ associated to simple roots $\alpha_i\in \Pi$.
Set
$$
w_j:= s_{j+1} \cdots s_n\in W
\qquad(1\leq j\leq n).
$$
Then the reduced expression of $w$ satisfies the condition
\begin{equation} \label{eq root positive}
w_j^{-1}\alpha_j \in \Sigma^+\qquad (1 \leq j \leq n)
\end{equation}
and
\begin{equation}\label{eq different roots}
w_j^{-1}\alpha_j\neq w_k^{-1}\alpha_k
\qquad(1\leq j<k\leq n).
\end{equation}
See \cite[VI.1.6 Corollaire 2]{Bourbaki}.
Essential for our reasoning is the factorization
\begin{equation}\label{eq factorization J}
J_{w,\lambda}= J_{s_1,w_1\lambda}\circ J_{s_2,w_2\lambda}\circ \cdots J_{s_{n-1},w_{n-1}\lambda}\circ  J_{s_n, \lambda},
\end{equation}
with each
$$J_{s_j, w_j\lambda}: V_{w_j\lambda}^\infty  \to V_{w_{j-1}\lambda}^\infty$$
a rank one intertwiner.

\subsection{Rank one intertwining operators}
Let $s_\alpha\in W$ be the reflection in a simple root $\alpha\in\Sigma^+$.
When restricted to a specific $K$-type $\tau\in \hat K_M$, each
$J_{s_\alpha, \lambda}[\tau]$ is an element of $\End(V_{\tau^\vee}^M)$ depending rationally on $\lambda$.
We will describe the entries of a diagonal matrix for it.

Let $\gf_\alpha$ be the semisimple rank one subalgebra of $\gf$ generated by $\alpha$. Then
$$
\gf_\alpha
= \gf^{2\alpha} \oplus \gf^{\alpha}\oplus\af_\alpha \oplus \mf_\alpha  \oplus \gf^{-\alpha} \oplus \gf^{-2\alpha}
$$
with $\af_\alpha = \R \alpha^\vee$ and $\mf_\alpha\triangleleft \mf$ an ideal.
In particular,  the Cartan decomposition $\gf=\kf \oplus\sf$ descends to $\gf_\alpha$, and we obtain with
$\kf_\alpha:=\gf_\alpha \cap \kf$ a maximal compact subalgebra of $\gf_\alpha$.
We denote by $G_\alpha:= \la \exp(\gf_\alpha)\ra$ the analytic subgroup of $G$ associated to $\gf_\alpha$, by
$K_\alpha:=\exp(\kf_\alpha)$ the maximal compact subgroup of $G_\alpha$ with Lie algebra $\kf_\alpha$, and by $M_\alpha$ the group $M\cap K_\alpha$.
Note that $M$ normalizes $G_\alpha$. Hence if we branch $V_{\tau^\vee}$ with respect to $K_\alpha$, then
$$
V_{\tau^\vee}^M
= \bigoplus_{\delta\in \hat {K_\alpha}_{M_\alpha}} m(\delta) V_\delta^{M_\alpha},
$$
where $m(\delta)$ denotes the multiplicity of $\delta$ in $\tau^{\vee}|_{K_\alpha}$.  As $K_\alpha/M_\alpha$ is symmetric, each $V_\delta^{M_\alpha}$ is one-dimensional. We choose an orthonormal basis (depending on $j$) of $V_{\tau^{\vee}}^{M}$ of vectors from these one-dimensional subspaces.

For elements $\frac{1}{2} \leq a\leq b$ with $a, b \in \frac{1}{2} \N$ and $b-a\in\N_{0}$, we recall the polynomials $\Gamma_{a,b}(z)$ from \eqref{def Gamma ab}.
With respect to the chosen basis the operator $J_{s_\alpha,\lambda}[\tau]$
is of diagonal form, say
$$
D_{\tau}(\lambda)
= \diag( d_\tau^{1} (\lambda), \ldots, d_\tau^{l(\tau)}(\lambda)),
$$
and each diagonal entry is of the form (see \eqref{factor tau},  Lemma \ref{Lemma q polynomials})
\begin{equation} \label{diag entry}
d_\tau^{k}(\lambda)
= \frac{\Gamma_{a, b_{k}}(- i\lambda(\frac{\alpha^{\vee}}{\gamma_\alpha}))
\Gamma_{a', b_{k}'}(- i\lambda(\frac{\alpha^{\vee}}{\gamma_\alpha}))}{\Gamma_{a, b_{k}}( i\lambda(\frac{\alpha^{\vee}}{\gamma_\alpha}))\Gamma_{a', b_{k}'}( i\lambda(\frac{\alpha^{\vee}}{\gamma_\alpha}))}\qquad (1\leq k \leq l(\tau)),
\end{equation}
where $\gamma_\alpha=2$ and $a'=b'_{k}$ if $m_{2\alpha}=0$, and otherwise $\gamma_\alpha=4$.
The parameters $a$ and $a'$ depend only on $\alpha$, and
for all $\tau\in\hat K$ the parameters $b_k$ and $b'_k$ satisfy
$$
 b_{k}'-a'\leq b_{k}-a\leq m |\tau| \qquad (1\leq k\leq l(\tau) )
$$
for some $m\in \N$ independent of $\tau$ and $\alpha$. Therefore we may and shall assume
that $b_k, b_k'\leq m|\tau|$ for all non-trivial $\tau$.

\subsection{Cancellation of poles and estimate}
Let $\alpha\in\Sigma^+$ be a simple root and let $\tau\in\hat K_M$.
In the following lemma we determine a polynomial on $\af_\C^*$ which
cancels the poles of $J_{s_\alpha, \lambda}[\tau]$. Moreover, we give an estimate
of the product of
$J_{s_\alpha, \lambda}[\tau]$ with this polynomial.

As in Section \ref{Section rank one cases} we write $\lceil\tau\rceil=\lceil|\tau|\rceil$.
We  define
the following polynomial on $\C$,
\begin{equation}\label{defi e_tau}
e_\tau(z)
:= \Gamma_{1, m\lceil\tau\rceil+1}(z)^{2} \,\Gamma_{\frac{1}{2},m\lceil\tau\rceil+\frac{1}{2}}(z)^{2},
\end{equation}
and recall the polynomials $q_{n}(z)$ from (\ref{def q_n}). In particular, we see from
\eqref{eq estimate Gamma_ab} that we can estimate the polynomial $e_\tau(z)$ by
\begin{equation}\label{eq estimate e_tau}
|e_\tau(z)| \leq (1+2\lceil\tau\rceil)^2 \,q_{m\lceil\tau\rceil}(|z|)^4
\end{equation}
for all $z\in \C$ and all $\tau\in\hat K$.

\begin{lemma} \label{lemma estimate J}
Let $\alpha\in\Sigma^+$ be simple.
\begin{enumerate}[\rm (i)]
\item \label{item 1 of lemma estimate J}
The map $\af_\C^*\to \End(V^M_{\tau^\vee})$ given by
$$
\lambda\mapsto e_\tau\big(i\lambda(\tfrac{\alpha^\vee}{\gamma_\alpha})\big) J_{s_\alpha, \lambda}[\tau]
$$
is polynomial  for every $\tau$ in $\hat K$.
\item \label{item 2 of lemma estimate J}
There exists a constant $C>0$ such that
$$
\|\,e_\tau\big(i\lambda(\tfrac{\alpha^{\vee}}{\gamma_\alpha})\big) J_{s_\alpha, \lambda}[\tau]\,\|_{\rm op}
\leq  C(1+|\tau|)^{4}\, q_{m\lceil\tau\rceil} \big(\big|\lambda(\tfrac{\alpha^{\vee}}{\gamma_\alpha})\big|\big)^8
$$
for every $\tau \in \hat K$ and $\lambda\in \af_\C^*$.
\end{enumerate}
\end{lemma}

\begin{proof} We may assume $|\tau|\neq 0$ since
$J_{s_\alpha, \lambda}[\tau]=1$ for the trivial $K$-type.
We fix a simple root $\alpha\in\Sigma^+$ and define polynomials $d_\tau^{k, +}$ and $d_\tau^{k, -}$ to be the numerator and denominator in
\eqref{diag entry}, respectively. Then
$$
d_\tau^{k}(\lambda)
=\frac{d_\tau^{k, +} (\lambda)}{d_\tau^{k, -}(\lambda)}.
$$

Next we make the following observation: $\Gamma_{a,b}(z)$ divides
$\Gamma_{a, b+n}(z)$ for all $n\in \N_0$, and $\Gamma_{a, b}(z)$ divides
$\Gamma_{a-n, b}(z)$ for all $n\in \N_0$ such that $a-n\geq \frac{1}{2}$.
It follows that
$$
d_{\tau}^{k, -} (\lambda) \big | e_\tau\big(i\lambda(\tfrac{\alpha^\vee}{\gamma_\alpha})\big)
$$
for all $k$, and this implies \eqref{item 1 of lemma estimate J}. Moreover, together with
(\ref{eq Gamma geq 1}) it implies
$$
\left|\frac{e_\tau\big(i\lambda(\tfrac{\alpha^\vee}{\gamma_\alpha})\big)}{d_{\tau}^{k, -} (\lambda)}\right|
\le
\frac{e_\tau\big(|\lambda(\tfrac{\alpha^\vee}{\gamma_\alpha})|\big)}{d_{\tau}^{k, -} (|\lambda|)}
\le
e_\tau\big(|\lambda(\tfrac{\alpha^\vee}{\gamma_\alpha})|\big)
$$
for all $\lambda\in  \af_\C^*$.
For the numerator $d_\tau^{k, +} (\lambda)$ we find
$$
\big|d_{\tau}^{k,+}(\lambda)\big|
\leq e_\tau\big(|\lambda(\tfrac{\alpha^\vee}{\gamma_\alpha})|\big)
$$
also for all $\lambda\in  \af_\C^*$.
Hence
$$
\big|e_\tau\big(i\lambda(\tfrac{\alpha^\vee}{\gamma_\alpha})\big) \,d_{\tau}^{k}(\lambda)\big|
\leq e_\tau\big(|\lambda(\tfrac{\alpha^\vee}{\gamma_\alpha})|\big)^2
$$
for all indices $k$. By \eqref{eq estimate e_tau} this implies \eqref{item 2 of lemma estimate J}.
\end{proof}

\subsection{Application of  Proposition \ref{prop basic estimate}}
The following lemma contains the main estimate for the proof of Theorem \ref{main theorem}.
Recall from \eqref{def f_n} the functions $f_{n,R}$.
We will determine an estimate for a product of these functions with $J_{w,\lambda}$.
For this we use Lemma \ref{lemma estimate J}, Proposition \ref{prop basic estimate} and
the factorization \eqref{eq factorization J}.

   Let
$$h:=8\pi\max_{\alpha\in\Sigma^+} \|\tfrac{\alpha^\vee}{\gamma_\alpha}\|\,.$$

 \begin{lemma}\label{lemma estimate J with w}
There exist $ c, R_0>0$ and for every $r>0$ a constant $C_{r}>0$  so that for every $R>R_0$ with
\begin{equation}\label{estimate with R and cr}
\frac{(\log R)^2}{R^2} <c r
\end{equation}
one has
\begin{align*}
\Big\|\Big(\prod_{\alpha\in\Sigma^+} f_{m\lceil\tau\rceil,R}\big(\lambda(\tfrac{\alpha^{\vee}}{\gamma_{\alpha}})\big)^8  \,
&e_\tau\big(i\lambda(\tfrac{\alpha^{\vee}}{\gamma_{\alpha}})\big)
\Big)\,
J_{w, \lambda}[\tau]\,\Big\|_{\mathrm{op}} \\
&\leq
C_r\,\Big[(1+|\tau|)^{4}\,e^{8mr|\tau|}e^{hR\,\|\!\im \lambda\|}\Big]^{|\Sigma^+|}
\end{align*}
for all $\tau\in\hat K$, $\lambda\in  \af_\C^*$ and $w\in W$.
\end{lemma}

\begin{proof} Let $ c$, $R_0$ be as in Proposition \ref{prop basic estimate} and
let $r>0$. We first show that there exists a constant $C_r>0$ such that
if $R>R_0$ satisfies \eqref{estimate with R and cr} then
\begin{equation}\label{estimate J for simple roots}
\big\|   f_{m\lceil\tau\rceil,R}\big(\lambda(\tfrac{\alpha^{\vee}}{\gamma_\alpha})\big)^8  \, e_\tau\big(i\lambda(\tfrac{\alpha^{\vee}}{\gamma_\alpha})\big)\,
J_{s_\alpha, \lambda}[\tau]\big\|_{\mathrm{op}}
\leq
C_{r}(1+|\tau|)^4\,e^{8mr|\tau|}e^{ hR\,\|\!\im \lambda\|}
\end{equation}
for all $\tau\in\hat K$, $\lambda\in  \af_\C^*$ and all simple roots $\alpha\in\Sigma^+$.

We apply Proposition \ref{prop basic estimate} with $n=m\lceil\tau\rceil$ and
$$P(z)=(1+|\tau|)^{-4}\, e_\tau(iz)\, J_{ s_\alpha, z\mu}[\tau],$$
where
$$\mu=\lambda(\tfrac{\alpha^{\vee}}{\gamma_\alpha})^{-1}\lambda\in  \af_\C^*.$$
The estimate in Lemma \ref{lemma estimate J} ensures the proposition is applicable.
Hence
\begin{equation*}
\big\|f_{m\lceil\tau\rceil,R}(z)^8  \, e_\tau(iz) \,
J_{s_\alpha, z\mu}[\tau]\big\|_{\mathrm{op}}
\leq
(1+|\tau|)^4 \, [C_r \, e^{rm\lceil\tau\rceil} e^{\pi R\,\|\!\im z\|}]^8.
\end{equation*}
By inserting
$z=\lambda(\tfrac{\alpha^{\vee}}{\gamma_\alpha})$ and $\lceil\tau\rceil\le |\tau|+1$
we obtain \eqref{estimate J for simple roots} for some $C_r>0$.

Let $w\in W$ and
consider the factorization \eqref{eq factorization J} of $J_{w,\lambda}$. By
submultiplicativity of the operator norm we obtain from \eqref{estimate J for simple roots}
that
\begin{align*}
\Big\|\Big(\prod_{j=1}^n  f_{m\lceil\tau\rceil,R}\big(\lambda(\tfrac{w_j^{-1}\alpha_j^{\vee}}{\gamma_{\alpha_j}})\big)^8  \,
&e_\tau\big(i\lambda(\tfrac{w_j^{-1}\alpha_j^{\vee}}{\gamma_{\alpha_j}})\big)\Big)\,
J_{w, \lambda}[\tau]\,\Big\|_{\mathrm{op}} \\
&\leq
\big[C_r(1+|\tau|)^{4}\,e^{8mr|\tau|}e^{hR\,\|\!\im \lambda\|}\big]^n.
\end{align*}
The $w_j^{-1}\alpha_j$ are all distinct and positive by \eqref{eq different roots} and \eqref{eq root positive}, respectively. Hence each factor of the above product over $j$ occurs exactly once in
$$
\prod_{\alpha\in\Sigma^+}  \big[f_{m\lceil\tau\rceil,R}\big(\lambda(\tfrac{\alpha^{\vee}}{\gamma_{\alpha}})\big)^8  \,
e_\tau\big(i\lambda(\tfrac{\alpha^{\vee}}{\gamma_{\alpha}})\big)
\big].
$$
On the other hand, since by \eqref{eq estimate e_tau}
the scalar valued polynomial $e_\tau$
satisfies the estimate
$$|e_\tau(z)|\leq (1+2|\tau|)^2 q_{m\lceil\tau\rceil}(|z|)^4
\leq  (1+2|\tau|)^4 q_{m\lceil\tau\rceil}(|z|)^8, $$
we obtain in analogy with \eqref{estimate J for simple roots}
that
$$\big|f_{m\lceil\tau\rceil,R}(\lambda(\tfrac{\alpha^{\vee}}{\gamma_{\alpha}}))^8 \,
e_\tau(\lambda(\tfrac{\alpha^{\vee}}{\gamma_{\alpha}}))\big|
\leq
C_r (1+|\tau|)^4 e^{8mr|\tau|} e^{hR\,\|\! \im\lambda\|}$$
for every $\alpha\in\Sigma^+$.
We apply this estimate to the roots $\alpha\in\Sigma^+$ which are not of the form $w_j^{-1}\alpha_j$ for any $j$ and obtain the estimate as stated in the lemma.
\end{proof}

\subsection{Conclusion of proof}

We can now give the proof of Theorem \ref{main theorem}, following
Ansatz 2 from Section \ref{subsection Ansatz 2}.
Recall that $\lambda_0$ satisfies \eqref{Kos}, that is,
$$\re(i\lambda_0(\alpha^\vee))\ge 0$$
for all $\alpha\in\Sigma^+$. 
We define the following functions on
$\af_\C^*$.
\begin{enumerate}[I.]
\item We  choose a polynomial $p_{\lambda_0}:\af_\C^*\to\C$ such that 
$$
\begin{cases}
p_{\lambda_{0}}(\lambda_{0})=\frac{1}{|W_{\lambda_0}|}\,,\\
p_{\lambda_{0}}(w\lambda_{0})=0\qquad(w\in W\setminus W_{\lambda_{0}})\,,
\end{cases}
$$
where $W_{\lambda_0}\subset W$ is the stabilizer of $\lambda_0$.
\item For each $\tau\in \hat K_M$ we define a polynomial
$$
p_{\tau}(\lambda)
:=\frac
{\prod_{\alpha\in\Sigma^{+}} e_\tau\big(i\lambda(\tfrac{\alpha^\vee}{\gamma_\alpha})\big)}
{\prod_{\alpha\in\Sigma^{+}} e_\tau\big(i\lambda_{0}(\tfrac{\alpha^\vee}{\gamma_\alpha})\big)}.
$$
It follows from \eqref{defi e_tau} and \eqref{eq Gamma geq 1} that
\begin{equation}\label{estimate below on e_tau}
|e_\tau\big(i\lambda_{0}(\tfrac{\alpha^\vee}{\gamma_\alpha})\big)|
\ge 1
\end{equation}
for all $\alpha\in\Sigma^{+}$.
\item For every
$R>0$ for which
\begin{equation}\label{eq condition on R}
\forall \alpha\in\Sigma^+:\quad R\lambda_{0}(\tfrac{\alpha^{\vee}}{\gamma_\alpha}) \notin \Z\bs\{0\},
\end{equation}
we define for each $n\in\N_0$
 an entire function on $\af_\C^*$ by
\begin{equation*}
\psi_{n,R}(\lambda)
:=\frac{\prod_{\alpha\in \Sigma^{+}} f_{n,R} \big(\lambda(\frac{\alpha^\vee}{\gamma_\alpha})\big)}{\prod_{\alpha\in \Sigma^{+}} f_{n,R} \big(\lambda_0(\frac{\alpha^\vee}{\gamma_\alpha})\big)}.
\end{equation*}
By \eqref{eq condition on R} and
Lemma \ref{lemma normalization f_n} there exists a constant $c_{R}>0$ so that
\begin{equation}\label{estimate below on f_{n,R}}
|f_{n,R}\big(\lambda_0(\tfrac{\alpha^\vee}{\gamma_\alpha})\big)| \ge c_{R}
\end{equation}
for all $n\in\N_0$ and $ \alpha\in\Sigma^+$.
\end{enumerate}
After these definitions we let
$$
\phi_\tau(\lambda)
:= p_\tau(\lambda)  [\psi_{m \lceil\tau\rceil, R}(\lambda)]^8\,
$$
for $\tau\in\hat{K}_{M}$, and we define $F_{\tau}:\af_{\C}^{*}\to\End(V_{\tau})$ by
$$
F_{\tau}(\lambda)
=\sum_{w\in W} \phi_{\tau}(w^{-1}\lambda)p_{\lambda_{0}}(w^{-1}\lambda)J_{w,w^{-1}\lambda}[\tau]
\qquad(\lambda\in\af_{\C}^{*})\,.
$$
We are going to apply Proposition \ref{Prop inverse PW} to  $F_\tau$, and for that we need to verify its
conditions (\ref{Prop inverse PW - assumption 1})-(\ref{Prop inverse PW - assumption 3}). As explained in Section \ref{subsection Ansatz 2}, condition (\ref{Prop inverse PW - assumption 1}) follows from
the fact that $v_{K,\lambda_0}$ is cyclic for $V_{\lambda_0}$ (see \eqref{property of Ansatz 2}), and
(\ref{Prop inverse PW - assumption 2})  is an automatic consequence of the cocycle condition
$$J_{w_2,w_1\lambda}\circ J_{w_1,\lambda}=J_{w_2w_1,\lambda}$$
for the intertwining operators.

Let $c,R_0$ be as in Lemma \ref{lemma estimate J with w}, and let $r>0$ and $R>R_0$
satisfy \eqref{estimate with R and cr}. Let $r'<r$ be such that
$$\frac{(\log R)^2}{R^2}<cr'.$$
By perturbing $R$ to a slightly smaller value we may assume that \eqref{eq condition on R} is valid.
It follows from Lemma \ref{lemma estimate J with w}
together with the denominator estimates \eqref{estimate below on e_tau}-\eqref{estimate below on f_{n,R}}
that there exists a constant $C>0$ so that for every $\lambda\in\af_{\C}^{*}$ and $\tau\in\hat{K}_{M}$
$$
\|F_{\tau}(\lambda)\|_{\mathrm{op}}
\leq C (1+|\tau|)^{4|\Sigma^+|}\,e^{ar' |\tau|} \,(1+\|\lambda\|)^{\deg p_{\lambda_0}}
e^{ AR \,\|\im \lambda\|}\, ,
$$
where $a=8m|\Sigma^{+}|$ and $A=h|\Sigma^{+}|$.
This gives the remaining condition (\ref{Prop inverse PW - assumption 3}) of Proposition \ref{Prop inverse PW}, and with that
can conclude
that there is a continuous embedding
$$V_{\lambda_0}^\omega(ar)\subset C^\infty_{AR+\epsilon}(G)* v_{K,\lambda_0}\, .$$
By Remark \ref{rmk scaling invariant} this implies the continuous embedding in Theorem \ref{main theorem}. Finally, as $v_{K,\lambda_{0}}$ is $\U(\gf)$-cyclic by Lemma \ref{lemma quotient}, we have $C^\infty_{R}(G)* v_{K,\lambda_0}=(V_{\lambda_{0}})^{\min}_{R}$.

\appendix
\section{The domains $\kf(R)$}\label{Appendix A}
We recall the open $\Ad(K)$-invariant domains $\kf(R)$, with $R>0$, from  \eqref{def k(R) 2}.
In this appendix we describe these in two interesting examples.

\subsection{The unit disc: $G=\SU(1,1)$}
While treating this example we use a notation so that the generalization
to general Hermitian symmetric spaces becomes straightforward. First note that
$G_\C= \SL(2,\C)$ acts transitively on the projective space $\P^1(\C)$. We identify $\P^{1}(\C)$ with $\C \cup\{\infty\}$ via the map
$$
\P^{1}(\C)\to\C \cup\{\infty\},\quad \C\left(
                                         \begin{array}{cc}
                                           z  \\
                                           1  \\
                                         \end{array}
                                       \right)
                                       \mapsto z.
$$
We define the subgroups of $G$
$$K=\left\{ k_\theta:=\begin{pmatrix} e^{i\theta} & 0 \\ 0 & e^{-i\theta}\end{pmatrix}\mid \theta \in \R\right\} \quad
\hbox{and} \quad A=\left\{ a_t = \begin{pmatrix}\cosh t  & \sinh t \\ \sinh t &\cosh t\end{pmatrix}\mid t\in \R\right\}$$
and note that $K_\C=\left\{ \begin{pmatrix} z & 0 \\ 0 & z^{-1}\end{pmatrix} \mid z\in \C^*\right\}$. Further
we  define unipotent abelian subgroups of $G_\C$ by
$$P^+:=\left\{ \begin{pmatrix} 1 & z\\ 0 & 1\end{pmatrix} \mid z\in \C\right\} \quad \hbox{and}
\quad P^-:=\left\{ \begin{pmatrix}  1& 0\\ z & 1\end{pmatrix} \mid z\in \C\right\} \,.$$
Note that $P^{+}$ and $P^{-}$ are the stabilizers of $\infty$ and $0$, respectively.
Then both $K_\C P^\pm$ are Borel subgroups of $G_\C$ with $K_\C P^+ \cap K_\C P^-=K_\C$. Hence,  $Z_\C= G_\C/ K_\C$ is realized as an open affine subvariety of
the projective variety $G_\C/ K_\C P^+ \times G_\C/ K_\C P^-$ via
$$ gK_\C \mapsto (gK_\C P^+, gK_\C P^-)\, .$$
In more concrete terms, if we identify $G_\C/ K_\C P^+ \times G_\C/ K_\C P^-$ with
$\P^1(\C)\times \P^1(\C)$  via
$$ G_\C/ K_\C P^+ \times G_\C/ K_\C P^-\to \P^1(\C)\times \P^1(\C)$$
$$ (g_1 K_\C P^+,  g_2K_\C P^-)\mapsto
(g_1^{-t} (0), g_2(0))\, ,$$
then $Z_\C$ is given by
$$Z_\C =\P^1(\C)\times  \P^1(\C)\setminus \{(z,w): w\neq \phi(z)\} \, .$$
where $\phi$ is the automorphism of $\P^1(\C)$ which is induced from the linear
map $\C^2\ni (z_1, z_2)\mapsto (-z_1, z_2)\in \C^2$.
\par Let us denote by $\D=\{ z\in \C\mid |z|<1\}$ the open unit disk (i.e. ~the bounded realization of $G/K$) and note that
$$Z=G/K=\{(z, \oline z): z\in \D\}\subset Z_\C\, .$$
Now one has that
that the crown domain is given by
$$\Xi=\D \times \D \subset Z_\C\, .$$
(A similar result holds for general Hermitian symmetric spaces, see \cite[Sect. 3]{BHH} or \cite[Th.~7.7]{K-S1}.)
For $R>0$ we note that
$$A_R=\left\{a_t \in A\mid |t|\leq R/\sqrt{8}\right\}\, .$$
Now we calculate
$$k_{i\theta}a_t\cdot z_0= (e^{2\theta} \tanh t , e^{-2\theta} \tanh t )\in Z_\C. $$
This is in $\Xi = \D \times \D$ precisely if $\theta\in (-r, r)$
for $r>0$ defined by $e^{2r}\tanh \frac{R}{\sqrt{8}}=1$. Thus we have shown:

\begin{prop} \label{prop SL} Let $G=\SU(1,1)$, $R>0$.
Then
$$ \kf(R)=\{ Y\in \kf\mid \|Y\|< \beta_R / \sqrt{8}\}\, ,$$
where  $\beta_R=\frac{1}{2}\log\big(\coth(\frac{R}{\sqrt{8}})\big)$.
\end{prop}

\subsection{The hyperboloids: $G=\SO_o(1,n)$}
Let $G=\SO_o(1,n)$ with $K=\SO(n,\R)$ being embedded into $G$ as the lower right corner. (The group $G$ does not satisfy the condition that it is the group of real points of a connected algebraic reductive group defined over $\R$. Instead one could consider the group  $SO(1,n)$, which would satisfy this condition, but for convenience of notation we rather work with its connected component.)
Consider the following quadratic form on $\C^{n+1}$
$$\square (u) = u_0^2 -u_1^2 -\ldots - u_n^2$$
and let $u\cdot v$ be the bilinear pairing obtained by polarization.
Then
$$Z=G/K=\{x\in\R^{n+1}\mid \square (x) = 1, x_0>0\}, $$
$$Z_\C=G_\C/K_\C=\{u\in\C^{n+1}\mid \square (u) = 1 \}$$
and
$$\Xi=\{ u=x+iy \in Z_\C \mid \square (x) >0, x_0>0  \}\, ,$$
see \cite[p.96]{Gi}. The canonical base point in $Z_\C$ is given by
$z_0=(1,0\ldots, 0)^T\in Z_\C$.
\par Set $l=\left[\frac{n}{2}\right]$ and note that $l$ is the
rank of $K$. Our choice and parametrization of $\tf$ are
as follows:
\begin{equation} \label{torus Lorentz}  \R^l \ni \beta=(\beta_1, \ldots, \beta_l)\mapsto T_\beta:=\diag(0, \beta_1 U,
\ldots, \beta_l U )\in \tf \end{equation}
where
$$
U=\left(
    \begin{array}{cc}
      0 & -1 \\
      1 & 0 \\
    \end{array}
  \right)
$$
and the first zero in the diagonal matrix means
the zero $1\times 1$-matrix in case $n$ is even and
the zero $2\times 2$-matrix if $n$ is odd.

With the standard choice of $A$  and $R':=R/ \sqrt{2(n-1)}$ we have
$$A_R=\left \{ \begin{pmatrix}
\cosh t &0 & \sinh t \\ 0 & {\bf 1} & 0\\
\sinh t &0 & \cosh t \end{pmatrix}\mid |t|\leq   R'\right\}
$$
and an easy computation yields
$$KA_R\cdot z_0=\left\{
\begin{pmatrix} \cosh t\\ u \end{pmatrix}\mid u\in \R^n, t\in
[-R', R'], \ \|u\|_2 = |\sinh t|\right\}\, .$$
In the sequel we only treat the case of $n=2l$ being even;  the odd case requires just a small modification.

With $k_\beta=\exp(iT_\beta)$
we obtain from \eqref{torus Lorentz} that
$$k_\beta \begin{pmatrix} \cosh t\\ u \end{pmatrix}=
\begin{pmatrix} \cosh t\\  u_1\cosh \beta_1  - iu_2\sinh \beta_1 \\
iu_1\sinh\beta_1  + u_2\cosh\beta_1\\ \vdots\end{pmatrix}\, .$$
The right hand side is now in the crown domain  if and only if
$$\square\Big(\re k_\beta \begin{pmatrix} \cosh t\\ u \end{pmatrix}\Big)=
\cosh^2 t- \cosh^2 \beta_1 (u_1^2 +u_2^2)- \ldots- \cosh^2 \beta_l (u_{n-1}^2+ u_n^2)>0\, .$$
There is no loss of generality in restricting our attention to the closure
$\tf^+$ of a chamber in $\tf$, i.e. we may assume that
$\beta_1\geq \beta_2\geq \ldots \geq \beta_{l-1}\geq|\beta_l|\geq 0$.
Then the condition from above for all $u$ with $\|u\|_2 =|\sinh t|$
means nothing else as
$$\cosh^2 t -(\cosh ^2 \beta_1) \sinh^2 t >0$$
for all $t\in [-R', R']$. A short calculation reformulates that in
$$|\sinh \beta_1|< \frac{1}{\sinh R'}\, .$$
We have thus shown:

\begin{prop} \label{prop SO}For $G=\SO_o (1,n)$, $R>0$  and the notation
introduced from above one has that
$$
\tf(R)^+
=\left\{ T_\beta\in \tf^+\mid|\sinh \beta_1|< \frac{1}{\sinh R'}\right\}\,
$$
where $R'=R/ \sqrt{2(n-1)}$.
\end{prop}

\section{The Helgason conjecture}\label{Appendix B}
In this appendix we briefly describe how (\ref{min is omega}) implies the Helgason conjecture. We are essentially following Schmid's approach from \cite{S}.

For $\lambda\in\af_{\C}^{*}$ we define
$$
\cP_{\lambda}:V_{\lambda}^{\infty}\to C^{\infty}(G/K),
\quad v\mapsto \Big(g \mapsto \int_{K}v(gk)\,dk\Big).
$$
This map admits a continuous extension to the space $V_{\lambda}^{-\omega}:=(V_{-\lambda}^{\omega})'$.
Let $\Diff(G/K)$ be the commutative algebra of $G$-invariant differential operators on $G/K$. As before, let  $v_{K,\lambda}$ be the $K$-fixed vector in $V_{\lambda}$ with $v_{K,\lambda}(e)=1$. Note that
\begin{equation}\label{eq Poisson transform as matrix coeff}
\cP_{\lambda}(v)(g)
=\langle g^{-1}\cdot v, v_{K,-\lambda} \rangle
\qquad(v\in V_{\lambda}).
\end{equation}
The algebra $\U(\gf)^{K}/\U(\gf)^{K}\cap\U(\gf)\kf$ acts from the right on smooth functions on $G/K$.  In fact  $\Diff(G/K)$ is isomorphic to $\U(\gf)^{K}/\U(\gf)^{K}\cap\U(\gf)\kf$. Note that $\U(\gf)^{K}$ acts by scalars on $\C v_{K,-\lambda}$, and hence $\Diff(G/K)$ acts by a character $\chi_{\lambda}$ on the image of $\cP_{\lambda}$.
We write $C^{\infty}(G/K)_{\lambda}$ for the space of joint eigenfunctions of $\Diff(G/K)$ with eigencharacter $\chi_{\lambda}$.

The following theorem is the Helgason conjecture, which was first proven in \cite{K6}.

\begin{theorem}\label{Thm Helgason conjecture}
Let $\lambda\in\af_{\C}^{*}$ be so that the $K$-spherical vector $v_{K,-\lambda}$ in $V_{-\lambda}$ is $\U(\gf)$-cyclic. Then $\cP_{\lambda}$ defines a $G$-equivariant isomorphism
\begin{equation}\label{eq Helgason conjecture}
V_{\lambda}^{-\omega}\to C^{\infty}(G/K)_{\lambda}
\end{equation}
of topological vector spaces.
\end{theorem}

\begin{rmk}
By Lemma \ref{lemma quotient} $v_{K,-\lambda}$  is $\U(\gf)$-cyclic if $-\lambda$ satisfies (\ref{Kos}).
\end{rmk}

We derive the theorem from (\ref{min is omega}). We recall Schmid's maximal globalization of a Harish-Chandra module $V$,
$$
V_{\max}
=\Hom_{(\gf,K)}\big(V^{\vee},C^{\infty}(G)\big),
$$
where $V^{\vee}$ is the dual Harish-Chandra module of $V$, i.e. the space of $K$-finite vectors in the algebraic dual of $V$. Further, $C^{\infty}(G)$ is considered as a $(\gf,K)$-module, where $\gf$ and $K$ act via the right-regular representation.
We provide $V_{\max}$ with a topology as follows. The space
\begin{equation}\label{eq def E}
E
:=\Hom_\C \big(V^\vee, C^\infty(G)\big)
\end{equation}
is a countable product of copies of the Fr{\'e}chet space $C^\infty(G)$ and hence is a Fr{\'e}chet space. Now $V_{\rm max} = \Hom_{(\gf, K)} (V^\vee, C^\infty(G))$ is a closed subspace and as such inherits the structure of a Fr{\'e}chet space.
Moreover, the $G$-action on $V_{\max}$ is continuous.

\begin{lemma}\label{Lemma Vmax reflexive} For any Harish-Chandra module $V$, the maximal globalization $V_{\rm max} $ is a reflexive
Fr\'echet space.
\end{lemma}

\begin{proof} First we recall that $C^\infty(G)$ is reflexive. As the space $E$ from (\ref{eq def E}) is a countable product of reflexive Fr{\'e}chet spaces, it is reflexive by \cite[Prop.~24.3]{MV}.  Now $V_{\rm max}$ is a closed subspace of  $E$ and as such reflexive by \cite[Prop.~23.26]{MV}.
\end{proof}

By taking matrix coefficients one sees that any globalization of $V$ embeds continuously into $V_{\max}$. Here by globalization we understand a completion of $V$ to a representation of $G$ on a complete Hausdorff topological vector space $E=\oline{V}$.  Note that the assignment $V\mapsto V_{\max}$ is functorial. We define $V^{-\omega}$ as the continuous dual of $(V^{\vee})^{\omega}$ equipped with the strong topology.

\begin{prop}\label{Prop V_max =V-omega}
For every Harish-Chandra module $V$ we have
$$
V_{\max}
=V^{-\omega}
$$
as topological $G$-modules.
\end{prop}

\begin{proof}
As $V_{\min}=V^\omega$ for all Harish-Chandra modules $V$,
it suffices to show that $V_{\max}=(V^\vee)_{\min}'$.

We recall from Lemma \ref{Lemma Vmax reflexive} that $V_{\max}$ is reflexive.
Since $V_{\max}'$ is a globalization of $V^\vee$ there exists an embedding $(V^\vee)_{\min}\to V_{\max}'$.
Taking duals we obtain a map $V_{\max}\to (V^\vee)_{\min}'$. On the other hand $ (V^\vee)_{\min}'$ is a globalization
of $V$ and hence embeds into $V_{\max}$. As these maps restrict to the identity on $V$, it
follows that $V_{\max}=(V^\vee)_{\min}'$ as asserted.
\end{proof}

\begin{prop}\label{Prop V_lambda= eigenfunctions}
Let $\lambda\in\af_{\C}^{*}$ be so that the $K$-spherical vector $v_{K,-\lambda}$ in $V_{-\lambda}$ is $\U(\gf)$-cyclic. Then
$$
(V_{\lambda})_{\max}
=C^{\infty}(G/K)_{\lambda}
$$
as topological $G$-modules.
\end{prop}

For the proof of the proposition we need the following lemma.

\begin{lemma}\label{Lemma formula for V-lambda}
Let $\lambda\in\af_{\C}^{*}$. If $v_{K,-\lambda}$ is $\U(\gf)$-cyclic in $V_{-\lambda}$, then
 $$
 V_{-\lambda}
 =\U(\gf)\otimes_{\U(\gf)\kf+\U(\gf)^{K}} \C v_{K,-\lambda}
 $$
 as $(\gf,K)$-modules.
 \end{lemma}

 \begin{proof}
By assumption the natural map $\U(\gf)\otimes_{\U(\gf)\kf+\U(\gf)^{K}} \C v_{K,-\lambda}\to V_{-\lambda}$ of $(\gf,K)$-modules is surjective. It remains to prove injectivity.
Recall from (\ref{eq U(g)-decomposition}) that $\U(\gf)=\Hc^{\star}(\sf) \Ic^{\star}(\sf) \oplus \U(\gf) \kf$. Since $\Ic^{\star}(\sf)=\Hc^{\star}(\sf) \Ic^{\star}(\sf) \cap \U(\gf)^{K}$, we have as $K$-modules
$$
\U(\gf)\otimes_{\U(\gf)\kf+\U(\gf)^{K}} \C v_{K,-\lambda}
=\Hc^{\star}(\sf)\Ic^{\star}(\sf)\otimes_{\Ic^{\star}(\sf)}\C v_{K,-\lambda}
\simeq \Hc^{\star}(\sf).
$$
By Kostant-Rallis \cite{KR} the right-hand side is $K$-isomorphic to $\C[K/M]$. Since $V_{-\lambda}$ is $K$-isomorphic to $\C[K/M]$ as well, the assertion follows from the finite dimensionality of the $K$-isotypes.
 \end{proof}

\begin{proof}[Proof of Proposition \ref{Prop V_lambda= eigenfunctions}]
By Lemma \ref{Lemma formula for V-lambda}, we have the following equalities of $G$-modules,
\begin{align*}
(V_{\lambda})_{\max}
&=\Hom_{(\gf, K)}\big(V_{-\lambda}, C^{\infty}(G)\big)\\
&=\Hom_{(\gf, K)}\big(\U(\gf)\otimes_{\U(\gf)\kf+\U(\gf)^{K}} \C v_{K,-\lambda}, C^{\infty}(G)\big)\\
&=\Hom_{(\U(\gf)\kf+\U(\gf)^{K},K)}\big(\C v_{K,-\lambda}, C^{\infty}(G)\big)\\
&=\Hom_{\U(\gf)\kf+\U(\gf)^{K}}\big(\C v_{K,-\lambda}, C^{\infty}(G/K)\big)\,.
\end{align*}
The assertion now follows from the definition of $C^{\infty}(G/K)_{\lambda}$.
\end{proof}

\begin{proof}[Proof of Theorem \ref{Thm Helgason conjecture}]
In view of Proposition \ref{Prop V_max =V-omega} and Proposition \ref{Prop V_lambda= eigenfunctions}, both sides of (\ref{eq Helgason conjecture}) are isomorphic to $(V_{\lambda})_{\max}$. Furthermore, as $v_{K,-\lambda}$ is $\U(\gf)$-cyclic, it follows from (\ref{eq Poisson transform as matrix coeff}) that $\cP_{\lambda}$ is injective, and hence bijective, on the space of $K$-finite vectors. The theorem now follows from the functoriality of the maximal globalizations.
\end{proof}

\section{An application to eigenfunctions on $Z=G/K$}\label{Appendix C}

We recall the crown domain $\Xi\subset Z_\C$, the natural $G$-extension of $Z$ inside of $Z_\C$.
Also we recall from the preceding appendix that  $V_{\lambda, {\rm max}} = C^\infty(Z)_\lambda$ for every
spherical principal series $V_\lambda$, $\lambda \in \af_\C^*$.
We mentioned in the introduction that for every $K$-spherical Harish-Chandra module $V$ with
$K$-spherical vector $v_K$ that the orbit map
$$f_v: G/K \to V^\infty, \ \ gK\mapsto g\cdot v_K$$
extends holomorphically to $\Xi$, see \cite[Th.~1.1]{K-S1}. Therefore, every
$\Diff(Z)$-eigenfunction extends holomorphically to $\Xi$, and thus we obtain that $C^\infty(Z)_\lambda= \Oc(\Xi)_\lambda$, i.e.
$$   V_{\lambda, {\rm max}}= \Oc(\Xi)_\lambda$$
by Prop. \ref{Prop V_lambda= eigenfunctions}.
Now for every $r>0$ we define $K$-invariant enlargements of $\Xi$ inside of $Z_\C$ by
$$ Z_\C(r):=  K_\C(r) \cdot \Xi= \exp(i\kf_r) \cdot \Xi \subset Z_\C\, .$$
It is not clear whether $Z_\C(r)$ is simply connected. Out of precaution we
pass to the simply connected cover $\tilde Z_\C(r)$ of $Z_\C(r)$. Note that $K$ acts naturally on the complex manifold
$\tilde Z_\C(r)$.  From the definition of $V_{\lambda,r}^\omega$ and $V_\lambda^\omega\subset V_{\lambda, \rm{max}}=
\Oc(\Xi)_\lambda$ we thus obtain
$$ V_{\lambda,r}^\omega= \Oc(\tilde Z_\C(r))_\lambda\, .$$
Hence the fact that $V_{\lambda, r}^\omega \subset V_{\lambda, \rm{min}}(R)$ for
$\frac{(\log R)^2}{R^2} < cr$ (see Theorem \ref{main theorem})
implies the following

\begin{theorem} \label{Theorem C}Let $-\lambda\in \af_\C^*$ satisfying \eqref{Kos} and $r, R>0$ such that $\frac{(\log R)^2}{R^2} < cr$.  Then any $f \in \Oc(\tilde Z_\C(r))_\lambda$ can be factorized
as
$$ f   =   \psi * \phi_\lambda$$
where $\phi_\lambda$ is the Harish-Chandra spherical function in $C^{\infty}(G/K)_{\lambda}$ and
$\psi \in C_R^\infty(G)$.
\end{theorem}

\end{document}